\pdfoutput=1 
\documentclass{article}
\usepackage{amssymb,amsmath,amsthm}
\usepackage[utf8]{inputenc}
\usepackage[T1]{fontenc}
\usepackage[bitstream-charter]{mathdesign}
\providecommand{\D}[1]{\,\mathrm{d}#1}
\usepackage{enumerate}
\usepackage{graphicx}
\usepackage{hyperref}
\usepackage{memhfixc}
\usepackage{algorithmic}

\DeclareMathOperator{\diag}{diag}
\DeclareMathOperator{\supp}{supp}
\DeclareMathOperator{\diam}{diam}

\newtheorem{lem}{Lemma}
\newtheorem{thm}{Theorem}

\newtheorem*{remark}{Remark}

\title{Multi-scale methods for wave propagation in heterogeneous media}
\author{Bj\"orn Engquist \and Henrik Holst \and Olof Runborg}
\date{2009-11-12}

\begin{document}
%%%%%%%%%%%%%%%%%%%%%%%%%%%%%%%%%%%%%%%%%%%%%%%%%%%%%%%%%%%%%%%%%%%%%%%%%%%%%%%%%%%%%%%%%%%%%%%%%%%%%%%%%%%%%%%%%%%%%%%%%%%%%%%%%%%%

\maketitle

\begin{abstract}
	Multi-scale wave propagation problems are computationally costly to solve by traditional techniques because the smallest
	scales must be represented over a domain determined by the largest scales of the problem.  We have developed and analyzed
	new numerical methods for multi-scale wave propagation in the framework of heterogeneous multi-scale method.  The numerical
	methods couples simulations on macro- and micro-scales for problems with rapidly oscillating coefficients. We show that the
	complexity of the new method is significantly lower than that of traditional techniques with a computational cost that is
	essentially independent of the micro-scale. A convergence proof is given and numerical results are presented for periodic
	problems in one, two and three dimensions. The method is also successfully applied to non-periodic problems and for long
	time integration where dispersive effects occur.
\end{abstract}

%%%%%%%%%%%%%%%%%%%%%%%%%%%%%%%%%%%%%%%%%%%%%%%%%%%%%%%%%%%%%%%%%%%%%%%%%%%%%%%%%%%%%%%%%%%%%%%%%%%%%%%%%%%%%%%%%%%%%%%%%%%%%%%%%%%%
\section{Introduction} \label{section:introduction}
%%%%%%%%%%%%%%%%%%%%%%%%%%%%%%%%%%%%%%%%%%%%%%%%%%%%%%%%%%%%%%%%%%%%%%%%%%%%%%%%%%%%%%%%%%%%%%%%%%%%%%%%%%%%%%%%%%%%%%%%%%%%%%%%%%%%

We consider the initial boundary value problem for the scalar wave equation,
\begin{equation}
	\begin{cases}
		u^{\varepsilon}_{tt} - \nabla \cdot A^{\varepsilon} \nabla u^{\varepsilon} = 0, & \Omega \times \{ 0 \leq t \leq T \}, \\
		u^{\varepsilon} = f, \quad u^{\varepsilon}_t = g, & \Omega \times \{ t = 0 \}, 
	\end{cases}
	\label{eq:introduction:wave}
\end{equation}
on a smooth domain $\Omega \subset \mathbb{R}^N$ with $A^{\varepsilon}(x)$ a symmetric uniformly positive definite matrix.  We
assume that $A^{\varepsilon}$ has oscillations on a scale proportional to $\varepsilon \ll 1$. The solution of
\eqref{eq:introduction:wave} will then also be highly oscillating in both time and spatial directions on the scale $\varepsilon$. It
is typically very computationally costly to solve these kinds of multi-scale problems by traditional numerical techniques. The
smallest scale must be well represented over a domain, which is determined by the largest scales.  For wave propagation small scales
may also originate from high frequencies in initial data or boundary data.  We will however focus on the case when they come from
strong variations in the wave velocity field.  Such variable velocity problems occur for example in seismic wave propagation in
subsurface domains with inhomogeneous material properties and microwave propagation in complex geometries.  

Recently, new frameworks for numerical multi-scale methods have been proposed, including the heterogeneous multi-scale
method~(HMM)~\cite{e2003} and the equation free methods~\cite{kevrekidis2003}.  These methods couple simulations on macro- and
micro-scales.  We use HMM, \cite{e2003,e2007,e2003b}, in which a numerical macro-scale method gets necessary information from
micro-scale models that are only solved on small sub domains. This framework has been applied to a number multi-scale problems, for
example, ODEs with multiple time scales~\cite{engquist2005}, elliptic and parabolic equations with multi-scale
coefficients~\cite{e2004,ming2005,abdulle2003}, kinetic schemes~\cite{e2007} and large scale MD simulation of gas
dynamics~\cite{li2005}. 

On the macro-scale we will assume a simple flux from,
\begin{equation} 
	\tilde{u}_{tt} - \nabla \cdot F = 0, 
\end{equation} 
in our HMM approximation of the wave equation \eqref{eq:introduction:wave}.  The solution $\tilde{u}$ should be a good approximation
of the solution to \eqref{eq:introduction:wave} and the value of $F$ on the macro-scale grid is computed by numerically
approximating \eqref{eq:introduction:wave} on small micro-scale domains. 

The goal of our research is to better understand the HMM process with wave propagation as example and also to derive computational
techniques for future practical wave equation applications. One contribution is a convergence proof in the multidimensional case
that includes a discussion on computational complexity. The analysis is partially based on the mathematical homogenization theory
for coefficients $A^{\varepsilon}$ with periodic oscillations~\cite{bensoussan1978,cioranescu1999}. 

Classical homogenization considers partial differential equations with rapidly oscillating coefficients.  As the period of the
coefficients in the PDE goes to zero, the solution approaches the solution to another PDE, a homogenized PDE. The coefficients in
the homogenized PDE has no $\varepsilon$ dependency.  For example, in the setting of composite materials consisting of two or more
mixed constituents (i.e., thin laminated layers $\varepsilon$ periodic), homogenization theory gives the macroscopic properties of
the composite.  It is an interesting remark that the macroscopic properties are often different than the average of the individual
constituents that makes up the composite~\cite{cioranescu1999}.  The wave equation~\eqref{eq:introduction:wave}, with
$A^{\varepsilon}(x) = A(x,x/\varepsilon)$ and $A(x,y)$ is periodic in $y$, have an homogenized equation,
\begin{equation}
	\begin{cases}
		\bar{u}_{tt} - \nabla \cdot \bar{A} \nabla \bar{u} = 0, & \Omega \times \{ 0 \leq t \leq T \}, \\
		\bar{u} = \bar{f}, \quad \bar{u}_t = \bar{g}, & \Omega \times \{ t = 0 \},
	\end{cases}
	\label{eq:introduction:wavebar}
\end{equation}
where $\bar{A}(x)$ is called the homogenized or effective coefficient.
The homogenized solution $\bar{u}$ can be used as an approximation of the solution $u^{\varepsilon}$ of the full equation since
$u^{\varepsilon}(x) = \bar{u}(x) + \mathcal{O}(\varepsilon)$.
%We refer to~\cite{bensoussan1978,nguetseng1989,allaire1993,cioranescu1999,jikov1991,pavliotis2007,marchenko2006,engquist2008} for
%more about homogenization in general.
Note that, the homogenized equations are often less expensive to solve with numerical methods, since the coefficients
varies slowly without $\varepsilon$ variations.
We refer to~\cite{bensoussan1978,nguetseng1989,cioranescu1999,jikov1991,marchenko2006,engquist2008} for
more about homogenization in general.

It should be noted that even if our numerical methods use ideas from homogenization theory they do not solve the homogenized
equations directly. The goal is to develop computational techniques that can be used when there is no known homogenized equation
available. In the research presented here many of the homogenized equations are actually available and could in practice be
numerically directly approximated.  We have chosen this case in order to be able to develop a rigorous convergence analysis and to
have a well-understood environment for numerical tests.  We also apply the techniques to problems that does not fit the theory. In
example \ref{section:example-three} an equation with non-periodic coefficients is approximated and in example
\ref{section:example-longtime} an equation is solved over very long time.  The latter is particularly interesting since the
homogenized solution contains dispersive effects, which influence the solution for $t=\mathcal{O}(\varepsilon^{-2})$. This
dispersive process is captured by a high accuracy HMM technique without explicit approximation of any dispersive term.

The article is organized as follows:
In section~\ref{section:hmm} we discuss first the HMM framework in a general setting and thereafter in
section~\ref{section:hmm_wave_equation} our HMM method for the wave equation. We give a rigorous proof of the approximation error by
the HMM method in the periodic coefficient case in section~\ref{section:convtheory}.  In section~\ref{section:results} we show
numerical results, which also includes a non-periodic problem and an example with very long time. The last
section~\ref{section:conclusions} ends this paper with our conclusions.

%%%%%%%%%%%%%%%%%%%%%%%%%%%%%%%%%%%%%%%%%%%%%%%%%%%%%%%%%%%%%%%%%%%%%%%%%%%%%%%%%%%%%%%%%%%%%%%%%%%%%%%%%%%%%%%%%%%%%%%%%%%%%%%%%%%%
\section{Heterogeneous multi-scale methods (HMM)} \label{section:hmm}
%%%%%%%%%%%%%%%%%%%%%%%%%%%%%%%%%%%%%%%%%%%%%%%%%%%%%%%%%%%%%%%%%%%%%%%%%%%%%%%%%%%%%%%%%%%%%%%%%%%%%%%%%%%%%%%%%%%%%%%%%%%%%%%%%%%%

In the HMM framework, the general setting of a multi-scale problem is the following: We assume that there exists two models, a
micro model $f(u,d) = 0$ describing the full problem and a coarse macro model $F(\tilde{u},\tilde{d}) = 0$.  The micro model is
accurate but is expensive to compute by traditional methods.  The macro model give a coarse scale or low frequency solution
$\tilde{u}$, assumed to be a good approximation of the micro-scale solution $u$ and is less expensive to compute.  The model is
however incomplete in some sense and requires additional data.  We assume that $F(\tilde{u},\tilde{d}) = 0$ can still be discretized
by a numerical method, called the macro solver.  A key idea in the HMM method is to provide the missing data in the macro model
($\tilde{d}$) using a local solution to the micro model. The micro model solution $u$ is computed locally on a small domain with
size proportional to the micro-scale. The initial data and boundary conditions ($d$) for this computation is constrained by the
macro-scale solution $\tilde{u}$.  

%%%%%%%%%%%%%%%%%%%%%%%%%%%%%%%%%%%%%%%%%%%%%%%%%%%%%%%%%%%%%%%%%%%%%%%%%%%%%%%%%%%%%%%%%%%%%%%%%%%%%%%%%%%%%%%%%%%%%%%%%%%%%%%%%%%%
\subsection{HMM for the wave equation} \label{section:hmm_wave_equation}
%%%%%%%%%%%%%%%%%%%%%%%%%%%%%%%%%%%%%%%%%%%%%%%%%%%%%%%%%%%%%%%%%%%%%%%%%%%%%%%%%%%%%%%%%%%%%%%%%%%%%%%%%%%%%%%%%%%%%%%%%%%%%%%%%%%%

We will formulate a general HMM framework for the wave equation on the domain $Y = [0,1]^d$.  
Let $u^{\varepsilon}$ be $Y$-periodic and solving,
\begin{equation}
	\begin{cases}
		u^{\varepsilon}_{tt} = \nabla \cdot A^{\varepsilon} \nabla u^{\varepsilon}, & Y \times \{ 0 \leq t \leq T \}, \\
		u^{\varepsilon} = f, \quad u^{\varepsilon}_t = g, & Y \times \{ t = 0 \}.
	\end{cases}
	\label{eq:hmm:wave}
\end{equation}
We follow the same strategy as in \cite{abdulle2003} for parabolic equations and in \cite{samaey2006} for the one-dimensional
advection equation. See also~\cite{engquist2007}.
%% ASSUMPTIONS
We assume there exists a macro-scale PDE of the form
\begin{equation} 
	\begin{cases}
		u_{tt} - \nabla \cdot F(x, u, \nabla u, \dots) = 0, & Y \times \{ 0 \leq t \leq T \},  \\
		u = f, \quad u_t = g, & Y \times \{ t = 0 \}, \\
		u, & Y \text{-periodic}.
	\end{cases}
	\label{eq:hmm:macroform}
\end{equation}
where $F$ is a function of $x$, $u$ and higher derivatives of $u$.  The assumption on \eqref{eq:hmm:macroform} is that $u \approx
u^{\varepsilon}$ when $\varepsilon$ is small.  In the clean homogenization case we would have $F = \bar{A} \nabla u$, but we will
not assume knowledge of a homogenized equation. Instead we will solve the PDE \eqref{eq:hmm:wave}, only in a small time and space
box, and from that solution extract a value for $F$. The form of the initial data for this micro problem will be determined from the
local behavior of $u$. In the method we suppose that $F = F(x, \nabla u)$.  

\paragraph*{Step 1: Macro model discretization.} 
We discretize \eqref{eq:hmm:macroform} using central differences with time step $K$ and spatial grid size $H$ in all directions,
\begin{equation}
	\begin{cases}
		U^{n+1}_m = 2 U^n_m - U^{n-1}_m + \frac{K^2}{H} \left( F^{(1)}_{m+\frac{1}{2} e_1} - F^{(1)}_{m-\frac{1}{2} e_1} \right) + \dots + \frac{K^2}{H} \left( F^{(d)}_{m+\frac{1}{2} e_d} - F^{(d)}_{m-\frac{1}{2} e_d} \right), \\
		F^n_{m-\frac{1}{2} e_k} = F(x_{m-\frac{1}{2} e_k},P^n_{m-\frac{1}{2} e_k}), \quad k = 1,\dots,d, \quad \text{(Note: $F^n_{m-\frac{1}{2} e_k}$ is a vector.)} \\
	\end{cases}
	\label{eq:hmm:macroscheme}
\end{equation}
where $F^n_{m \pm 1/2 e_k}$ is $F$ evaluated at point $x_{m \pm 1/2 e_k}$.  The quantity $P^n_{m \pm \frac{1}{2} e_k}$ approximates
$\nabla u$ in the point $x_{m \pm 1/2}$. We show an example in Figure \ref{fig:htwodisc5} of the numerical scheme in two dimensions.
There $P^n_{m+\frac{1}{2} e_2}$ is given by the expression \eqref{eq:P_in_2d} in the Appendix.
\begin{figure}[tbp]
	\centering
	\includegraphics[width=0.67\linewidth]{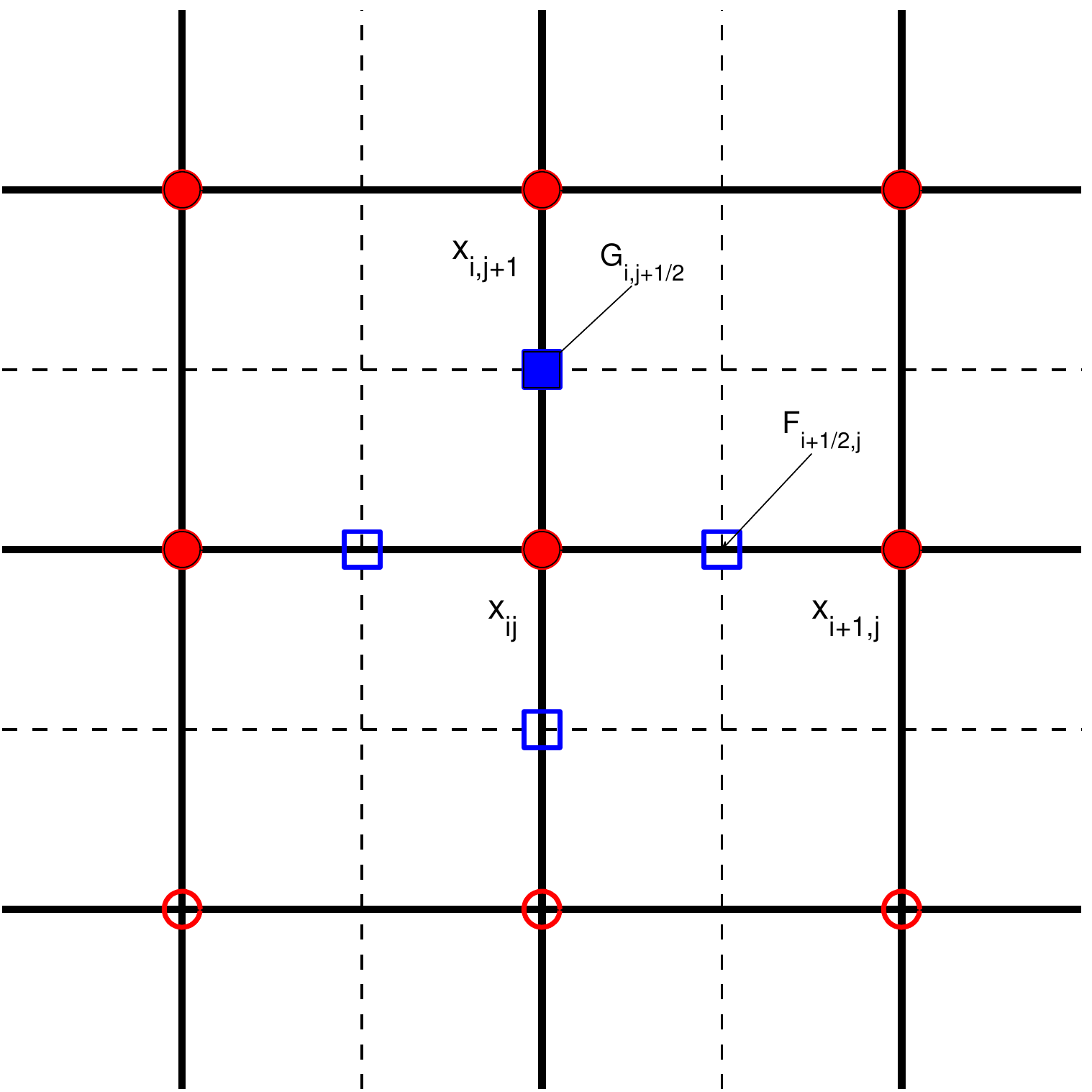}
	\label{fig:htwodisc5}
	\caption{The Numerical scheme \eqref{eq:P_in_2d} for $P$ in two dimensions. In the figure above the two components of $F$ in
	two different positions are given by $F_{i+1/2,j}$ and $G_{i,j+1/2}$. The $U$ points involved in computing
	$F^n_{m+\frac{1}{2}e_2}=G_{i,j+1/2}$ and $\nabla u \approx P^n_{m+\frac{1}{2}e_2}$ are indicated by filled circles.}
\end{figure}	
	
%% STEP 2

\paragraph*{Step 2: Micro problem.} 
The evaluation of $F^n_{m-\frac{1}{2} e_k}$ in each grid point is done by solving a micro problem to fill in the missing data in the
macro model.  Given the parameters $x_{m-\frac{1}{2} e_k}$ and $P^{n}_{m-\frac{1}{2} e_k}$, we solve a corresponding micro problem
over a small micro box $Y^{\varepsilon}$, centered around $x_{m-\frac{1}{2} e_k}$.
In order to simplify the notation, we make a change of variables $x - x_{m - \frac{1}{2} e_k} \mapsto x$.
This implies that $A^{\varepsilon}(x) \mapsto A^{\varepsilon}(x + x_{m - \frac{1}{2} e_k})$.
The micro problem has the form,
\begin{equation}
	\begin{cases}
		u^{\varepsilon}_{tt} - \nabla \cdot A^{\varepsilon} \nabla u^{\varepsilon} = 0, & Y^{\varepsilon} \times \{ 0 \leq t \leq \tau \}, \\
		u^{\varepsilon}(x,0) = (P^n_{m-\frac{1}{2} e_k}) \cdot x, \quad u^{\varepsilon}_t(x,0) = 0, & Y^{\varepsilon} \times \{ t = 0 \}, \\
		u^{\varepsilon} - u^{\varepsilon}(x,0), & \text{$Y^{\varepsilon}$-periodic.}
	\end{cases}
	\label{eq:hmm:microproblem}
\end{equation}
We keep the micro box size of order $\varepsilon$, i.e. $\tau$, $\diam Y^{\varepsilon} = \mathcal{O}(\varepsilon)$.  
We note that the solution $u^{\varepsilon}$ is an even function with respect to $t$ (i.e. $u^{\varepsilon}(x,-t) = u^{\varepsilon}(x,t)$)
due to the initial condition $u^{\varepsilon}_t(x,0) = 0$.

%% STEP 3

\paragraph*{Step 3: Reconstruction step.} 
After we have solved for for $u^{\varepsilon}$ for all $Y^{\varepsilon} \times [0, \tau]$ we approximate $F^n_{m-\frac{1}{2} e_k}
\approx \tilde{F}(x_{m-\frac{1}{2} e_k}, P^n_{m-\frac{1}{2} e_k})$.  The function $\tilde{F}$ is the mean value of $f^{\varepsilon}
= A^{\varepsilon} \nabla u^{\varepsilon}$ over $[-\eta,\eta]^d \times [-\tau,\tau]$ where $[-\eta,\eta]^d \subset Y^{\varepsilon}$.
The approximation can be improved with respect to the size of $\tau/\varepsilon$ and $\eta/\varepsilon$, by computing a weighted
average of $f^{\varepsilon}$.  We consider kernels $K$ described in \cite{engquist2005}: We let $\mathbb{K}^{p,q}$ denote the kernel
space of functions $K$ such that $K \in C^{q}_{c}(\mathbb{R})$ with $\supp \; K = [-1,1]$ and 
\begin{equation*}
	\int K(t) t^r \D{t} 
	= \begin{cases}
		1, & r = 0; \\
		0, & 1 \leq r \leq p.
	\end{cases}
	\label{eq:hmm:kernel}
\end{equation*}
Furthermore we will denote $K_{\eta}$ as a scaling of $K$
\begin{equation*}
	K_{\eta}(x) := \frac{1}{\eta}K\left(\frac{x}{\eta}\right), \qquad K \in \mathbb{K}^{p,q},
	\label{eq:hmm:kernelscaled}
\end{equation*}
with compact support in $[-\eta,\eta]$.
We use kernels of this sort to improve the approximation quality for the mean value computation,
\begin{equation}
	\tilde{F}(x_{m-\frac{1}{2} e_k}, P^n_{m-\frac{1}{2} e_k}) = \iint K_{\tau}(t) K_{\eta}(x) f^{\varepsilon}_k \D{x} \D{t}, \qquad f^{\varepsilon} = A^{\varepsilon}(x+x_{m+\frac{1}{2} e_k}) \nabla u^{\varepsilon},
	\label{eq:hmm:Ftildekernel}
\end{equation}
where here the multi variable kernel $K_{\eta}(x)$ is defined as
\begin{equation}
	K_{\eta}(x) = K_{\eta}(x_1) K_{\eta}(x_2) \cdots K_{\eta}(x_d),
\end{equation}
using the single valued kernel $K_{\eta}$, still denoted by $K_{\eta}$.  The domain $Y^{\varepsilon}$ is chosen such that
$[-\eta,\eta]^d \subset Y^{\varepsilon}$ and sufficiently large for information not to propagate into the region $[-\eta,\eta]^d$.
Typically we use 
\begin{equation}
	Y = [-y_{\text{max}},y_{\text{max}}]^d, \qquad y_{\text{max}} = \eta + \tau \sqrt{\sup \|A^{\varepsilon}\|_2},
	\label{eq:hmm:ysaftybox}
\end{equation}
c.f. discussion about micro solver boundary conditions in \cite{samaey2006}.
In this way we do not need to worry about the effects of boundary conditions.
Note therefore that other types of boundary conditions could also be used in \eqref{eq:hmm:microproblem}.

\begin{remark} \label{remark:kernexp}
	It is possible to find functions with infinite $q$.
	In \cite{engquist2005} a kernel $K_{\exp}$ is given, where $p=1$ and $q$ is infinite:
	\begin{equation*}
		K_{\exp}(x) = \begin{cases}
			C_0 \exp\left(\frac{5}{x^2-1}\right), & |x| < 1, \\
			0, & |x| \geq 1, 
		\end{cases}
	\end{equation*}
	where $C_0$ is chosen such that $\int K_{\text{exp}}(x) \D{x} = 1$.  
	This kernel is suitable for problems where $A^{\varepsilon}$ is of the form $A^{\varepsilon}(x) = A(x/\varepsilon)$.
\end{remark}

\begin{remark}
	The weighted integrals above are computed numerically with a simple trapezoidal rule.  
\end{remark}

\begin{remark}
	In our implementation, the micro problem \eqref{eq:hmm:microproblem} is solved with the same numerical scheme as the macro
	problem \eqref{eq:hmm:macroscheme}.
\end{remark}

%Interesting idea suggested by Daniel Appel\"o: Use grid points suitable for Gaussian quadrature.

%%%%%%%%%%%%%%%%%%%%%%%%%%%%%%%%%%%%%%%%%%%%%%%%%%%%%%%%%%%%%%%%%%%%%%%%%%%%%%%%%%%%%%%%%%%%%%%%%%%%%%%%%%%%%%%%%%%%%%%%%%%%%%%%%%%%
\subsection{Computational cost} \label{section:hmm:complexity}
%%%%%%%%%%%%%%%%%%%%%%%%%%%%%%%%%%%%%%%%%%%%%%%%%%%%%%%%%%%%%%%%%%%%%%%%%%%%%%%%%%%%%%%%%%%%%%%%%%%%%%%%%%%%%%%%%%%%%%%%%%%%%%%%%%%%

Let us assume that the time step is proportional to $\varepsilon$ in all direct solvers.  Using a direct
solver for \eqref{eq:hmm:wave} on the full domain implies a cost of order $\varepsilon^{-(d+1)}$.
The total cost for HMM is of the form $\left(\text{cost of micro problem}\right) \times M_d$ where $M_d$ is the number of micro
problems needed to be solved per macro time step.  The cost of a single micro problem is of the form $\left(\tau/\varepsilon\right)
\times \left(\eta/\varepsilon\right)^d$. We assume kernels with $\tau, \eta \sim \varepsilon$ and that $M_d$ does not depend on
$\varepsilon$. With these assumption our HMM method has a computational cost independent of $\varepsilon$.
%% precomputing
The constant can, however, still be large. Fortunately the computational cost of the HMM process can be reduced significantly. We
observe that the function \eqref{eq:hmm:Ftildekernel} is linear in $p$. It is in fact composed of three linear operations:
\begin{enumerate}
	\item Compute initial data $u(x,0)$ and $u_t(x,0)$ from $p$,
		$u(x,0) = p \cdot x$. 
	\item Solve $u^{\varepsilon}_{tt} - \nabla \cdot A^{\varepsilon} \nabla u^{\varepsilon} =0$
		for $0 \leq t \leq \tau$.
	\item Compute average $\tilde{F} = \iint K_{\tau} K_{\eta} f^{\varepsilon} \D{x} \D{t}$
		where $f^{\varepsilon} = A^{\varepsilon} \nabla u^{\varepsilon}$
\end{enumerate}
The first operation is clearly a linear operation. 
In step two we compute a
solution to a linear PDE, therefore this step is linear as well.  Computing the integral average in step three is also a linear
operation.  

As a corollary we can apply the HMM process to a smaller number of micro problems and form linear combinations of those for any given
$\tilde{F}$ computation.  More precisely, after precomputing $F(x,e_i)$, $i=1,2,\dots,d$
we can compute $\tilde{F}$ for fixed $x \in \Omega$ and any $p \in \mathbb{R}^d$,
\begin{equation}
	\tilde{F}(x, p) = \sum_{i=1}^d p_i F(x, e_i),
	\label{eq:canonical_micro_problems}
\end{equation}
where $p_i$ is the $i$th coefficient in $p$ in the basis $e_1, e_2, \dots, e_d$.  In conclusion, by precomputing the micro problems
$F(x_m,e_i)$ in \eqref{eq:canonical_micro_problems} we only need to solve $d$ micro problems in each macro grid point $x_m = m H$.
There is no need to solve any micro problems again in the next macro time step. 
The complexity is as before $\mathcal{O}(1)$, but with a lower constant not depending on the number of time step.
\begin{remark}
	In fact, if $A^{\varepsilon}$ is $\varepsilon$-periodic and the macro grid is such that $x_m = r (\mod \varepsilon)$, where
	$r$ is constant and independent of $m$, we only need to solve $d$ micro problems in total. In this case,
	the total cost is independent of both $\varepsilon$ and the macro grid size $H$.
\end{remark}

%%%%%%%%%%%%%%%%%%%%%%%%%%%%%%%%%%%%%%%%%%%%%%%%%%%%%%%%%%%%%%%%%%%%%%%%%%%%%%%%%%%%%%%%%%%%%%%%%%%%%%%%%%%%%%%%%%%%%%%%%%%%%%%%%%%%
\section{Convergence theory} \label{section:convtheory}
%%%%%%%%%%%%%%%%%%%%%%%%%%%%%%%%%%%%%%%%%%%%%%%%%%%%%%%%%%%%%%%%%%%%%%%%%%%%%%%%%%%%%%%%%%%%%%%%%%%%%%%%%%%%%%%%%%%%%%%%%%%%%%%%%%%%

In this section we apply the HMM process to the problem \eqref{eq:introduction:wave} with $A^{\varepsilon}(x) = A(x/\varepsilon)$
where $A$ is a $Y$-periodic symmetric positive matrix and show that it generates results close to a direct discretization of the
homogenized equation \eqref{eq:introduction:wavebar}.  In particular we show that
\begin{equation}
	\tilde{F}(x,p) = F(x,p) + \mathcal{O}\left(\left(\frac{\varepsilon}{\eta}\right)^{q}\right).
\end{equation}
The function $\tilde{F}$ and $F$ are defined in \eqref{eq:hmm:Ftildekernel} and \eqref{eq:hmm:macroform} respectively and we note that
here $F(x,p) = \bar{A} p$.  The integer $q$ depends on the smoothness of the kernel used to compute the weighted average of
$f^{\varepsilon}$ in \eqref{eq:hmm:Ftildekernel}.

We will formulate the problem in the setting of elliptic operators.  For the analysis we solve the micro problem
\eqref{eq:hmm:microproblem} over all of $\mathbb{R}^d$
\begin{equation}
	\begin{cases}
		u^{\varepsilon}_{tt} - \nabla \cdot A^{\varepsilon} \nabla u^{\varepsilon} = 0, & \mathbb{R}^d \times \{ 0 \leq t \leq \tau \}, \\
		u^{\varepsilon} = p \cdot x, \quad u^{\varepsilon}_t = 0, & \mathbb{R}^d \times \{ t = 0 \}.
	\end{cases}
	\label{eq:conv-nd:wave}
\end{equation}
Note that this gives the same $\tilde{F}$  as in \eqref{eq:hmm:Ftildekernel} if we choose a sufficiently large box $Y^{\varepsilon}$.

\begin{thm} \label{thm:conv}

	Let $\tilde{F}(x_0,p)$ be defined by \eqref{eq:hmm:Ftildekernel} where $u^{\varepsilon}$ solves the micro problem
	\eqref{eq:conv-nd:wave}, $A^{\varepsilon}(x) = A(x/\varepsilon)$ and $A$ is $Y$-periodic and smooth.  Moreover
	suppose $K \in \mathbb{K}^{p',q}$, $f$ and $g$ is smooth and $\tau = \eta$.  Then for $p\neq{}0$,
	\begin{equation*}
		\frac{1}{p} \left| \tilde{F}(x_0,p) - F(x_0,p) \right| \leq C \left(\frac{\varepsilon}{\eta}\right)^q,
	\end{equation*}
	where $C$ is independent of $\varepsilon$, $\eta$, $p$ and $q$. Furthermore, for the numerical approximation given in
	\eqref{eq:hmm:macroscheme} in one dimension, with $H = n \varepsilon$ for some integer $n$ and smooth initial data, we have the error estimate 
	\begin{equation*}
		|U^n_m - \bar{u}(x_m,t_n)| \leq C(T) \left( H^2 + (\varepsilon/\eta)^q \right), \qquad 0 \leq t_n \leq T,
	\end{equation*} 
	where $\bar{u}$ is the homogenized solution to \eqref{eq:introduction:wavebar}.  

\end{thm}

\begin{proof}
We will prove the Theorem in the following steps:
\begin{enumerate}
	\item Reformulate the problem as a PDE for a periodic function.
	\item Define an elliptic operator $L(y)$.
	\item Expand $\nabla_y \cdot A(y)$ and $v(y,t)$ (to be defined) in eigenfunctions to $L(y)$.
	\item Compute time dependent $v_j(t)$ coefficients in the above eigenfunction expansion.
	\item Compute the integral of $f^{\varepsilon}$ to get $\hat{F}$.
	\item Compute the solution to a cell problem and give final estimate.
\end{enumerate}

\paragraph*{Step 1:}
Express the solution to \eqref{eq:conv-nd:wave} as 
\begin{equation}
	u^{\varepsilon}(t,x) = p \cdot x + v(x/\varepsilon,t).
	\label{eq:conv-nd:umapstov}
\end{equation}
We insert this into \eqref{eq:conv-nd:wave} to get a PDE for $v$ 
\begin{equation}
	\begin{cases}
		v_{tt} = \frac{1}{\varepsilon} \nabla_y \cdot A(y) p + \frac{1}{\varepsilon^2} \nabla_y \cdot A(y) \nabla_y v(y), \\
		v(x,0) = 0, \quad v_t(x,0) = 0, \\
	\end{cases}
	\label{eq:conv-nd:vpde}
\end{equation}
where $y = x/\varepsilon$. Since $A$ is $Y$-periodic, so is $v$, and we can solve \eqref{eq:conv-nd:vpde} as a $Y$-periodic problem.

\paragraph*{Step 2:}
We define the linear operator $L(y) := -\nabla_y \cdot A(y) \nabla_y$ on $Y$ with periodic boundary conditions.  Denote by $w_j(y)$
the eigenfunctions and $\lambda_j$ the corresponding (non-negative) eigenvalue of $L$.  Since $L$ is uniformly elliptic, standard
theory on periodic elliptic operators informs us that all eigenvalues are strictly positive, bounded away from zero, except for the
single zero eigenvalue \cite{krein1948}
\begin{equation}
	0 = \lambda_0 < \lambda_1 < \lambda_2 \leq \cdots
\end{equation}
and $w_j \in C^{\infty}$ forms an orthonormal basis for $L^2_{\text{per}}(Y)$. 
Note also that $w_0 = |Y|^{-1}$ is a constant function.

\paragraph*{Step 3:}
We express $\nabla_y \cdot A(y)$ and $v(y)$ in eigenfunctions of $L$:
\begin{equation}
	\nabla_y \cdot A(y) = \sum_{j=1}^{\infty} a_j w_j(y) \quad \text{and} \quad v(y,t) = \sum_{j=0}^{\infty} v_j(t) w_j(y).
	\label{eq:conv-nd:eexp}
\end{equation}
Note that here $a_j$ are column vectors and as in the one dimensional case we have that $a_0 = 0$ since the mean value of $\nabla_y
\cdot A(y)$ is zero,
\begin{equation}
	a_0 = \int_{Y} \nabla_y \cdot A(y) w_0(y) \D{y} = \frac{1}{|Y|} \int_{Y} \nabla_y \cdot A(y) \D{y} = 0.
\end{equation}

\paragraph*{Step 4:}
We plug the eigenfunction expansions \eqref{eq:conv-nd:eexp} into \eqref{eq:conv-nd:vpde} and find that
\begin{align}
	\sum_{j=0}^{\infty} v''_j w_j = \frac{1}{\varepsilon} p \cdot \sum_{j=1}^{\infty} a_j w_j - \frac{1}{\varepsilon^2} \sum_{j=1}^{\infty} L v_j w_j = \sum_{j=1}^{\infty} \frac{p \cdot a_j}{\varepsilon} w_j - \sum_{j=1}^{\infty} \frac{\lambda_j}{\varepsilon^2} v_j w_j.
\end{align}
By collecting terms of $w_j$ we get
\begin{equation}
	v''_j + \frac{\lambda_j}{\varepsilon^2} v_j = \frac{p \cdot a_j}{\varepsilon}.
\end{equation}
This is a system of ODE:s similar to the form,
\begin{equation}
	y'' + \alpha y = \beta,
\end{equation}
which has the solution of the form ($\alpha>0$)
\begin{equation}
	y(t) = A e^{i t \sqrt{\alpha}} + B e^{-i t \sqrt{\alpha}} + \frac{\beta}{\alpha}.
\end{equation}
Note that all $\lambda_j>0$ ($j>0$) so it is known that the $v_j$ functions in the
problem have the form,
\begin{equation}
	v_j(t) = A_j e^{\frac{i t \sqrt{\lambda_j}}{\varepsilon}} + B_j e^{\frac{-i t \sqrt{\lambda_j}}{\varepsilon}} + r_j, \qquad r_j = \frac{\varepsilon p \cdot a_j}{\lambda_j},
	\label{eq:conv-nd:vjgen}
\end{equation}
and the special $v_0$ is given by
\begin{equation}
	v_0(t) = \frac{p \cdot a_0}{2 \varepsilon} t^2 + C t + D = C t + D \quad \text{since $a_0=0$}.
\end{equation}
By plugging the general solution \eqref{eq:conv-nd:vjgen} into the initial conditions of \eqref{eq:conv-nd:vpde}, we can formulate equations for $A_j$ and $B_j$ ($j>0$),
\begin{align}
	& v(0,x) = 0 \Rightarrow \sum_{j=0}^{\infty} v_j (0) w^{\varepsilon}_j(x) = 0 \Rightarrow v_j(0) = 0 \Rightarrow A_j + B_j + r_j = 0; \label{eq:conv-nd:vj1} \\
	& v_t(0,x) = 0 \Rightarrow \sum_{j=0}^{\infty} v'_j(0) w^{\varepsilon}_j(x) = 0 \Rightarrow v_j'(0) = 0 \Rightarrow \frac{i \sqrt{\lambda_j}}{\varepsilon} A_j - \frac{i \sqrt{\lambda_j}}{\varepsilon} B_j = 0.
	\label{eq:conv-nd:vj2}	
\end{align}
Similary, for $v_0(t)$
\begin{equation}
	v_0(0) = 0 \Rightarrow C = 0, \quad \text{and} \quad v'_0(0) = 0 \Rightarrow D = 0,
\end{equation}
thus $v_0(t) \equiv 0$. We solve for $A_j$ and $B_j$ and get
\begin{equation}
	A_j = B_j = \frac{r_j}{2} = -\frac{\varepsilon p \cdot a_j}{2 \lambda_j}, \qquad j = 1,2,\ldots
\end{equation}
All in all, the $v_j(t)$ coefficients in explicit form are
\begin{equation}
	\begin{cases}
		v_0(t) &= 0, \\
		v_j(t) &= -\frac{\varepsilon p \cdot a_j}{2 \lambda_j} \left( e^{\frac{i t \sqrt{\lambda_j}}{\varepsilon}} + e^{\frac{-i t \sqrt{\lambda_j}}{\varepsilon}}\right) + \frac{\varepsilon p \cdot a_j}{\lambda_j} 
		= \frac{\varepsilon p \cdot a_j}{\lambda_j} \left( 1 - \cos \frac{t \sqrt{\lambda_j}}{\varepsilon}\right) \qquad j=1,2,\ldots 
	\end{cases}
\end{equation}
The solution to our problem \eqref{eq:conv-nd:vpde} can then be expressed as
\begin{equation}
	v(y,t) = \varepsilon p \cdot \sum_{j=1}^{\infty} \frac{a_j}{\lambda_j} \left(1 - \cos \frac{t \sqrt{\lambda_j}}{\varepsilon}\right) w_j(y).
	\label{eq:conv-nd:ueps_eigen}
\end{equation}

\paragraph*{Step 5:} Now plug the expression \eqref{eq:conv-nd:ueps_eigen} into the expression \eqref{eq:conv-nd:umapstov} %$f^{\varepsilon} = A(x/\varepsilon) u^{\varepsilon}_x$
\begin{gather}
	\begin{split}
		f^{\varepsilon}= \nabla \cdot u^{\varepsilon} A(x/\varepsilon) = p \cdot \left(1+\sum_{j=1}^{\infty} \frac{a_j}{\lambda_j} \left(1-\cos \frac{t \sqrt{\lambda_j}}{\varepsilon}\right) \nabla_y \cdot w_j(x/\varepsilon) \right) A(x/\varepsilon).
	\end{split}
\end{gather}
We write down and analyze the function $f^{\varepsilon}$ in two parts $f^{\varepsilon} = p \cdot (\Lambda_1 + \Lambda_2)$, where
\begin{equation}
	\begin{cases}
		\Lambda_1(x/\varepsilon)   =   \left( I + \sum_{j=1}^{\infty} \frac{a_j}{\lambda_j} \nabla_y \cdot w_j(x/\varepsilon) \right) A(x/\varepsilon) , \\
		\Lambda_2(x/\varepsilon,t) = - \sum_{j=1}^{\infty} \frac{a_j}{\lambda_j} \cos \frac{t \sqrt{\lambda_j}}{\varepsilon} \nabla_y \cdot w_j(x/\varepsilon) A(x/\varepsilon).
	\end{cases}
\end{equation}

\paragraph*{Step 6a:}
First we show that $\Lambda_1 = \bar{A}$. 
To do that we need to use the so-called cell problem (or corrector problem, see Section~4.5.4 in \cite{evans1998}),
\begin{equation}
	\begin{cases}
		L(y) \chi = - \nabla_y \cdot A, & Y, \\
		\chi & \text{$Y$-periodic}.
	\end{cases}
	\label{eq:conv:cellproblem}
\end{equation}
We rewrite the cell problem \eqref{eq:conv:cellproblem} using a eigenfunctions expansion
\begin{equation}
	\sum_{j=0}^{\infty} \lambda_j\chi_j w_j = - \sum_{j=1}^{\infty} a_j w_j \Rightarrow \chi_j = - \frac{a_j}{\lambda_j}, \qquad j=1,2,\ldots,
\end{equation}
where $\chi_j$ are column vectors with coefficients of $\chi$ in the eigenfunctions expansion.
For the other term $\Lambda_1$ we now will make good use of the eigenfunction expansion of the cell solution $\chi$,
\begin{gather}
	\begin{split}
		\iint K_{\tau}(t) K_{\eta}(x) \Lambda_1 \D{x} \D{t}
		& = \int K_{\eta}(x) \left( I + \sum_{j=1}^{\infty} \frac{a_j}{\lambda_j} \nabla_y w_j(x/\varepsilon) \right) A(x/\varepsilon) \D{x} \\
		& = \int K_{\eta}(x) \left( I - \nabla_y \chi(x/\varepsilon) \right) A(x/\varepsilon) \D{x} \\
		& = \int K_{\eta}(x) \left( A(x/\varepsilon) - \nabla_y \chi(x/\varepsilon) A(x/\varepsilon) \right) \D{x} \\
		& = \int K_{\eta}(x) \left( A(x/\varepsilon) - A(x/\varepsilon) \nabla_y \chi(x/\varepsilon) \right) \D{x} \\		
		& = \bar{A} + \mathcal{O}\left(\left(\frac{\varepsilon}{\eta}\right)^q\right),
	\end{split}
	\label{eq:conv-nd:xavg}
\end{gather}
where we used Lemma \ref{lem:kernel}, in each coordinate direction.

\paragraph*{Step 6b:}
Now we should show that
\begin{equation}
	\iint K_{\tau}(t) K_{\eta}(x-x_0) \Lambda_2 \D{t} \rightarrow 0, \qquad \frac{\tau}{\varepsilon} \rightarrow \infty.
\end{equation}
For that we need a Lemma from \cite{engquist2005}:
\begin{lem}
	Let $f^{\varepsilon}(t) = f(t,t/\varepsilon)$, where $f(t,s)$ is $1$-periodic in the 
	second variable	and $\partial^r f(t,s) / \partial t^r$ is continuous for 
	$r = 0, 1, \ldots, p-1$. For any $K \in \mathbb{K}^{p,q}$ there exists constants 
	$C_1$ and $C_2$, independent of $\varepsilon$ and $\eta$, such that
	\begin{equation*}
		E = |K_{\eta} \ast f^{\varepsilon}(t) - \bar{f}(t)| \leq C_1 \eta^p + C_2 \left(\frac{\varepsilon}{\eta}\right)^q,
		\qquad \bar{f}(t) = \int_0^1 f(t,s) \D{s}.
	\end{equation*}
	If $f=f(t/\varepsilon)$ then we can take $C_1=0$.
	Furthermore, the error is minimized if $\eta$ is chosen to scale with $\varepsilon^{q/(p+q)}$.
	\label{lem:kernel}
\end{lem}
We now apply Lemma \ref{lem:kernel} to obtain
\begin{equation}
	\left| \int K_{\tau}(t) \cos \frac{t \sqrt{\lambda_j}}{\varepsilon} \D{t} \right| \leq 
	C_2 \left(\frac{2 \pi \varepsilon}{\sqrt{\lambda_j} \tau}\right)^q = C' \frac{1}{\lambda_j^{q/2}}  \left(\frac{\varepsilon}{\tau}\right)^q
	%\leq C' \frac{1}{\lambda_1^{q/2}}  \left(\frac{\varepsilon}{\tau}\right)^q
	\label{eq:conv-nd:bjbound}
\end{equation}
%
%%% NEW PROOF
%
Let $b_j$ and the column vector $g(y)$ be defined as
\begin{equation}
	b_j = \int K_{\tau}(t) \cos \frac{t \sqrt{\lambda_j}}{\varepsilon} \D{t}, \quad g(y) = \sum_{j=1}^{\infty} b_j \chi_j w_j(y),
\end{equation}
where we again used the solution to the cell problem \eqref{eq:conv:cellproblem} in the formulation of $g(y)$.
We then express $\iint K_{\tau} K_{\eta} \Lambda_2 \D{x} \D{t}$ using $g$, followed by a change of variables:
\begin{align*}
	\iint K_{\tau}(t) K_{\eta}(x-x_0) \Lambda_2 \D{x} \D{t}
	&= - \int K_{\eta} \nabla_y \cdot g(x/\varepsilon) A(x/\varepsilon) \D{x}  \\
	&= - \int K(x) \nabla_y \cdot \left(\frac{\eta x + x_0}{\varepsilon}\right) A\left(\frac{\eta x + x_0}{\varepsilon}\right) \D{x} \\
	&= - \int \frac{\varepsilon}{\eta} K(x) A\left(\frac{\eta x + x_0}{\varepsilon}\right) \nabla \cdot g\left(\frac{\eta x + x_0}{\varepsilon}\right) \D{x}.
\end{align*}
By doing integration by parts, using $K(e_k)=K(-e_k)=0$ ($k=1,2,\ldots,d$), together with Cauchy-Schwartz inequality, we obtain
\begin{multline}
	\int\limits_{[-1,1]^d} \underbrace{\nabla \cdot \left( \frac{\varepsilon}{\eta} K(x) A\left(\frac{\eta x + x_0}{\varepsilon}\right) \right)}_{\mathcal{O}(1) \text{ column vector}} g\left(\frac{\eta x + x_0}{\varepsilon}\right) \D{x} \\
	\leq \left( \int\limits_{[-1,1]^d} \left( \nabla \cdot \left( \frac{\varepsilon}{\eta} K(x) A\left(\frac{\eta x + x_0}{\varepsilon}\right) \right) \right)^2 \D{x}
	            \int\limits_{[-1,1]^d} g^2\left(\frac{\eta x + x_0}{\varepsilon}\right) \D{x}
	             \right)^{1/2}
\end{multline}
which is bounded by $C \|g\|_{(L^2_{\text{per}})^d}$ where $C$ is independent of $\varepsilon$, and $\eta$.
Finally we need to show that $\|g\| \rightarrow 0$. This is done by observing that
\begin{equation}
	\|g\|^2_{(L^2_{\text{per}})^d} = \sum_{j=1}^{\infty} b_j^2 \chi_j^2 \leq b^2_{\max} \sum_{j=1}^{\infty} \chi_j^2 = b^2_{\max} \|\chi\|^2_{(L^2_{\text{per}})^d},    
\end{equation}
where $|b_{\max}|$ is bounded by,
\begin{equation}
	\frac{C'}{\lambda_1^{q/2}}  \left(\frac{\varepsilon}{\tau}\right)^q,
\end{equation}
following the computations in \eqref{eq:conv-nd:bjbound}.
Then finally, we add our results from the calculations above and get,
\begin{gather}
	\begin{split}
		\hat{F}(x_0,p) &= p \cdot \iint K_{\tau}(t) K_{\eta}(x-x_0) f^{\varepsilon} \D{x} \D{t} \\
		& = p \cdot \iint K_{\tau}(t) K_{\eta}(x-x_0) ( \Lambda_1(t) + \Lambda_2(x/\varepsilon,t) ) \D{x} \D{t} \\
		& = p \cdot \left( \bar{A} + \mathcal{O}\left(\left(\frac{\varepsilon}{\eta}\right)^q\right) \right).
	\end{split}
	\label{eq:conv:result}
\end{gather}
This proves the Theorem.

\paragraph*{Final step:} 
Now we show the error estimate $|U^n_m - \bar{u}(x_m,t_n)| \leq C(T)(H^2 + (\varepsilon/\eta)^q)$.  We observe that $\tilde{F}$ in
the Theorem is of the form
\begin{equation}
	\tilde{F} = \tilde{A}(x) p
\end{equation}
where $\tilde{A}$ is $\varepsilon$-periodic.  By \eqref{eq:conv:result},
\begin{equation}
	|\tilde{A}(x) - \bar{A}| \leq C \left(\frac{\varepsilon}{\eta}\right)^q.
\end{equation}
By choosing $H=n \varepsilon$ for some integer $n$, we find that the macro scheme \eqref{eq:hmm:macroscheme} is a standard second
order discretization of the problem
\begin{equation}
	\begin{cases}
		u_{tt} - \tilde{A}(0) u_{xx} = 0, & \Omega \times \{ 0 \leq t \leq \tau \}, \\
		u(0,x) = f(x), \quad u_t = g,     & \Omega \times \{ t=0 \},
	\end{cases}
\end{equation}
since $\tilde{A}(x_m)$ = $\tilde{A}(m n \varepsilon)$ = $\tilde{A}(0)$ for all $m$.  Hence, if $g=0$ (the result is true also for $g
\neq 0$),
\begin{equation}
	u^n_m = \frac{1}{2} \left( f(x_m - \sqrt{\tilde{A}(0)} t_n) + f(x_m + \sqrt{\tilde{A}(0)} t_n) \right) + \mathcal{O}(H^2).
\end{equation}
On the other hand, the solution of the homogenized \eqref{eq:introduction:wavebar} with $g=0$ is
\begin{equation}
	\bar{u}(x_m,t_n) = \frac{1}{2} \left( f(x_m - \sqrt{\bar{A}} t_n) + f(x_m + \sqrt{\bar{A}} t_n) \right).
\end{equation}
Therefore we get the error estimate
\begin{align}
	|U^n_m - \bar{u}(x_m,t_n)| &\leq \sup_{|t| \leq T} \left| f(x+\sqrt{\tilde{A}(0)} t) - f(x+\sqrt{\bar{A}} t) \right| + C(T) H^2 \\
		&\leq |f'|_{\infty} T |\sqrt{\tilde{A}} - \sqrt{\bar{A}} | + C(T) H^2 \\
		&\leq C(T) \left( H^2 + (\varepsilon/\eta)^q \right),
\end{align}
for $0 \leq t_n \leq T$.
This proves the Theorem.
\end{proof}

%%%%%%%%%%%%%%%%%%%%%%%%%%%%%%%%%%%%%%%%%%%%%%%%%%%%%%%%%%%%%%%%%%%%%%%%%%%%%%%%%%%%%%%%%%%%%%%%%%%%%%%%%%%%%%%%%%%%%%%%%%%%%%%%%%%%
\section{Numerical results} \label{section:results}
%%%%%%%%%%%%%%%%%%%%%%%%%%%%%%%%%%%%%%%%%%%%%%%%%%%%%%%%%%%%%%%%%%%%%%%%%%%%%%%%%%%%%%%%%%%%%%%%%%%%%%%%%%%%%%%%%%%%%%%%%%%%%%%%%%%%

In this section we show numerical results when applying the HMM process to various problems in one, two and three dimensions.  The
notation in the experiments in the $d$-dimensional setting ($d=1,2,3$) is the following: We let $Y = [0,1]^d$ denote the macro domain
and $\varepsilon$ be the micro problem scale. We denote by $H$ and $K$ the macro grid size and time step respectively and for the
micro-scale we denote by $h$ and $k$ the grid size and time step respectively. We use explicit second order accurate finite
difference schemes (see Appendix).

%%%%%%%%%%%%%%%%%%%%%%%%%%%%%%%%%%%%%%%%%%%%%%%%%%%%%%%%%%%%%%%%%%%%%%%%%%%%%%%%%%%%%%%%%%%%%%%%%%%%%%%%%%%%%%%%%%%%%%%%%%%%%%%%%%%%
\subsection{Convergence study of different kernels} \label{section:results:kernels}
%%%%%%%%%%%%%%%%%%%%%%%%%%%%%%%%%%%%%%%%%%%%%%%%%%%%%%%%%%%%%%%%%%%%%%%%%%%%%%%%%%%%%%%%%%%%%%%%%%%%%%%%%%%%%%%%%%%%%%%%%%%%%%%%%%%%

In Figure~\ref{fig:kernconv} and Figure~\ref{fig:kernconv_varying} we present convergence results for the flux $F$ in terms of
$\eta/\varepsilon$. We use different type of kernels for the problem \eqref{eq:nr_1d_pde} with $A^{\varepsilon}(x) =
A_1(x/\varepsilon)$ and $A^{\varepsilon}(x) = A_2(x,x/\varepsilon)$ where $A_1(y) = 1.1 + \sin (2 \pi y)$ and $A_2(x,y) = 1.1 +
\frac{1}{2}(\sin 2 \pi x + \sin 2 \pi y)$. We compare our numerical results to the theoretical bounds in Theorem \ref{thm:conv}.  On
problems with both fast and slow scales which is not directly covered by Theorem~\ref{thm:conv}, we see a (slow) growth of the error
as $\tau, \eta \rightarrow \infty$ consistent with the general approximation result in Lemma~\ref{lem:kernel}.  We plot
$(\varepsilon/\eta)^{q}$ and $\eta^{p}$ separate with dashed lines. 

\begin{figure}[tbp]
	\centering
	\includegraphics{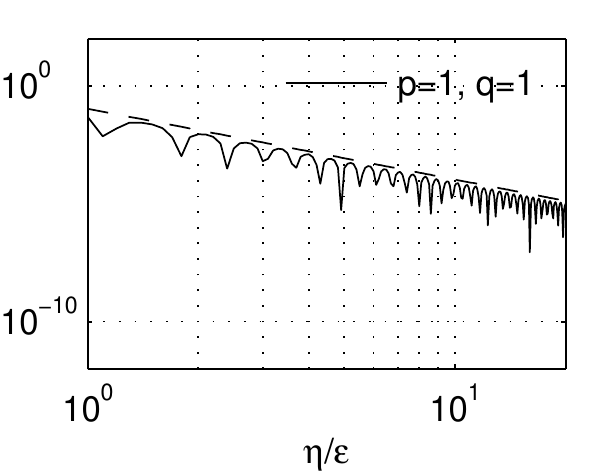}
	\includegraphics{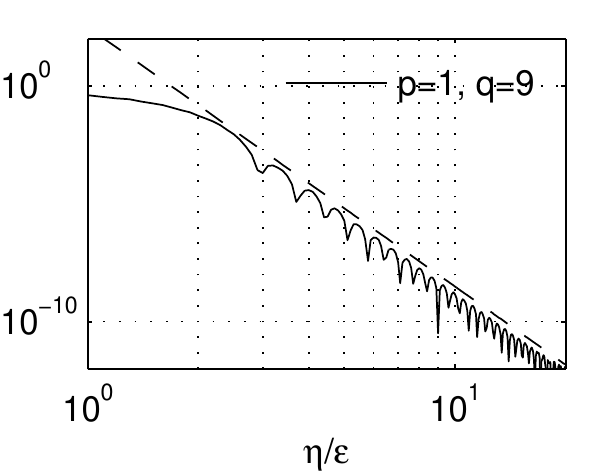}
	\includegraphics{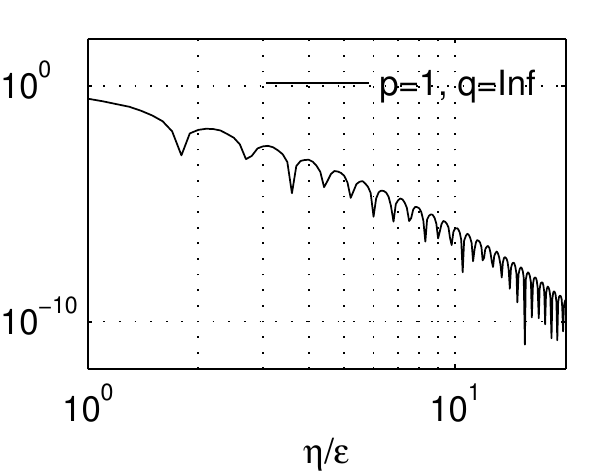}
	\caption{Convergence results, error $|\tilde{F} - \bar{F}|$ plotted against $\eta/\varepsilon$ ($\tau=\eta$) for fixed
	$\varepsilon = 0.01$ and where $A^{\varepsilon} = A_1$ with only fast scales. The dashed line corresponds to the
	$(\varepsilon/\eta)^{q}$ term in Theorem \ref{thm:conv}. The bottom figure shows results for the exponential kernel (see
	Remark \ref{remark:kernexp}) and indicates super algebraic convergence rate.}
	\label{fig:kernconv}
\end{figure}

\begin{figure}[tbp]
	\begin{center}
		\includegraphics{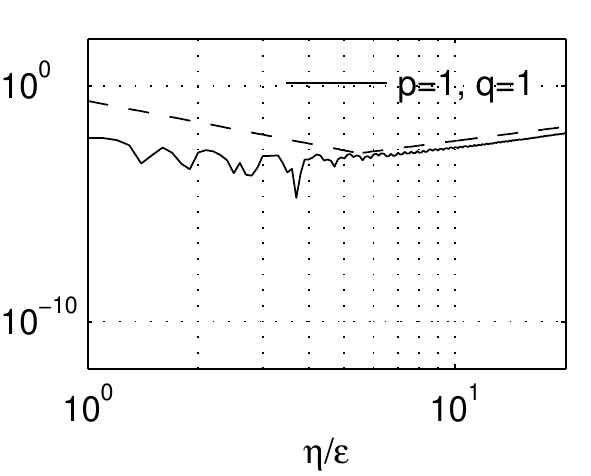}
		\includegraphics{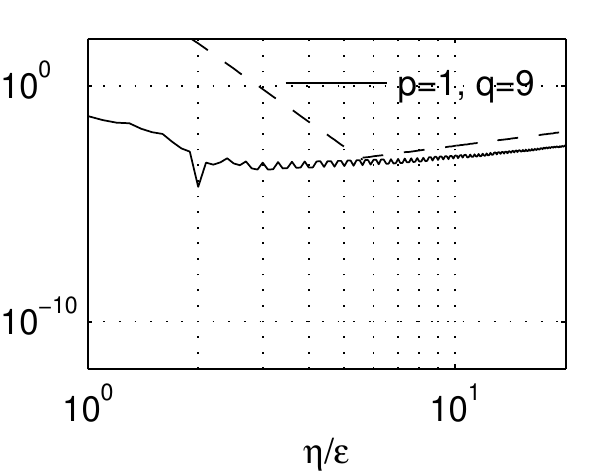}
		\includegraphics{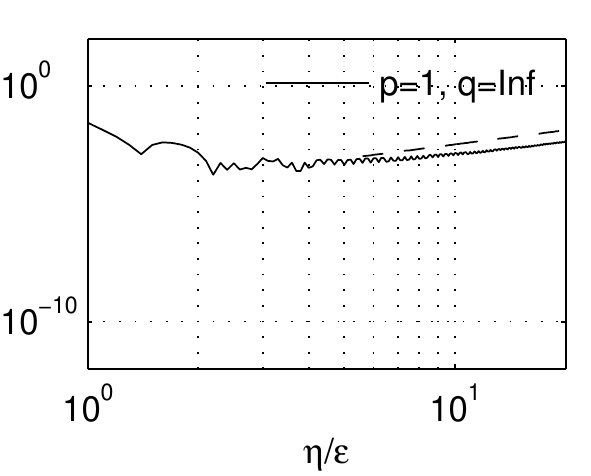}
	\end{center}
	\caption{Convergence results, error $|\tilde{F} - \bar{F}|$ plotted against $\eta/\varepsilon$ ($\tau=\eta$) for fixed
	$\varepsilon = 0.01$ and where $A^{\varepsilon} = A_2$ with both fast and slow scales. The dashed line with negative slope
	corresponds to the theoretical bound from the first term in Lemma 1 and the dashed line with positive slope corresponds to
	the $\eta^p$ term.}
	\label{fig:kernconv_varying}
\end{figure}

%%%%%%%%%%%%%%%%%%%%%%%%%%%%%%%%%%%%%%%%%%%%%%%%%%%%%%%%%%%%%%%%%%%%%%%%%%%%%%%%%%%%%%%%%%%%%%%%%%%%%%%%%%%%%%%%%%%%%%%%%%%%%%%%%%%%
\subsection{1D results}
%%%%%%%%%%%%%%%%%%%%%%%%%%%%%%%%%%%%%%%%%%%%%%%%%%%%%%%%%%%%%%%%%%%%%%%%%%%%%%%%%%%%%%%%%%%%%%%%%%%%%%%%%%%%%%%%%%%%%%%%%%%%%%%%%%%%

The general form for the one-dimensional examples is:
\begin{equation}
	\begin{cases}	
		u^{\varepsilon}_{tt} = \partial_x A^{\varepsilon} u^{\varepsilon}_x, & Y \times \{ 0 \leq t \leq T \}, \\
		u^{\varepsilon} = f, \quad u^{\varepsilon}_t = 0, & Y \times \{ t = 0 \},
	\end{cases}
	\label{eq:nr_1d_pde}
\end{equation}
where $Y = [0,1]$. We show some dynamics in Figure \ref{fig:nr_1d_dynamics} where we solved \eqref{eq:nr_1d_pde} for the
$A^{\varepsilon}$ and $f$ given in example one below.
\begin{figure}[tbp]
	\centering
	\includegraphics{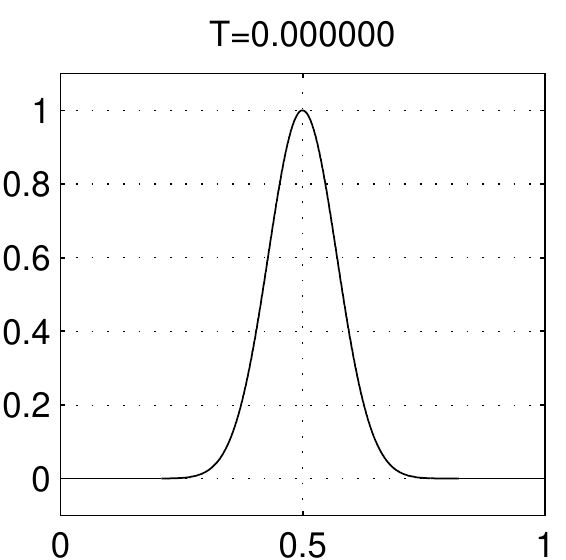}
	\includegraphics{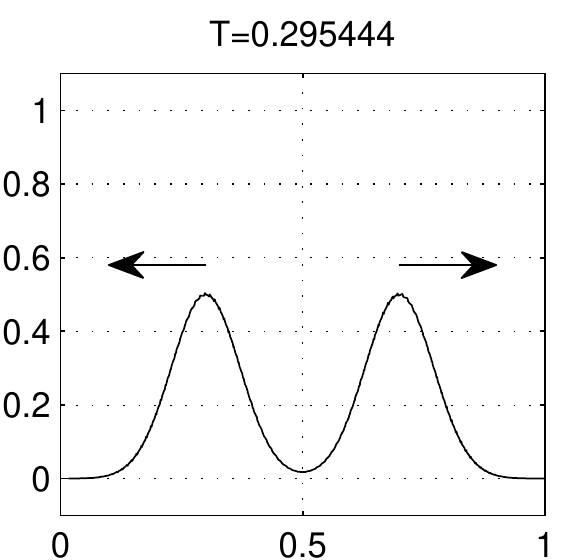}
	\includegraphics{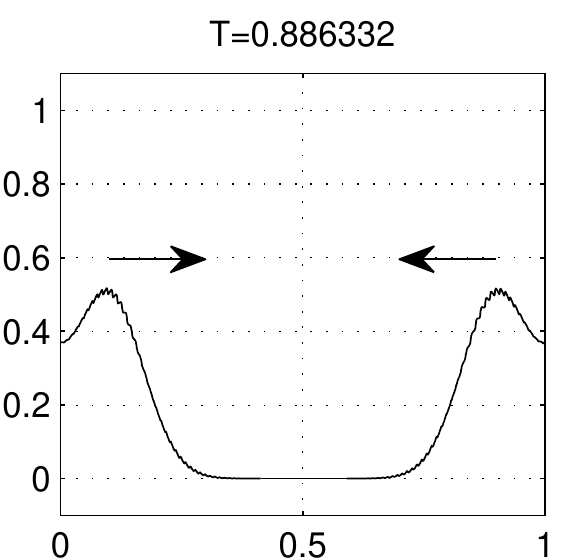}
	\includegraphics{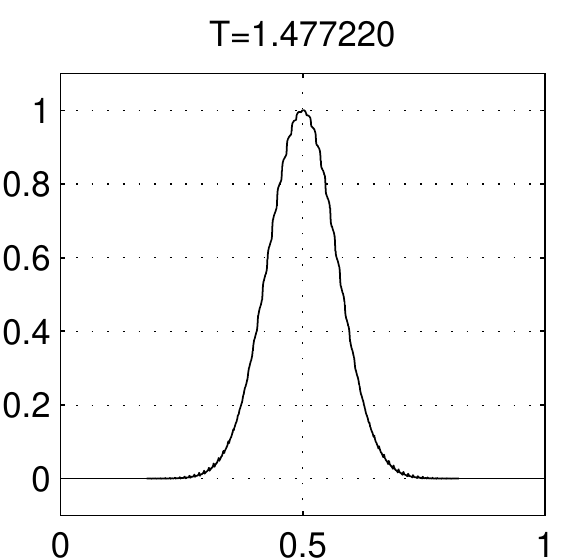}
	\caption{The dynamics of the problem \eqref{eq:nr_1d_pde} for 4 snapshots at $t_i = 3/(i \sqrt{\bar{A}}), 0 \leq i \leq 3$.
		Note the small oscillations superimposed on the smooth profile.
		We observe how the initial pulse separates in one left going and one right going pulse. 
		The effect of the periodic boundary condition can seen as the waves pass each other at the boundaries between frames 2 and 3.}
	\label{fig:nr_1d_dynamics}
\end{figure}
The homogenized solution to \eqref{eq:nr_1d_pde} will be of the form
\begin{equation}
	\begin{cases}
		\bar{u}_{tt} = \partial_x \bar{A} \bar{u}_x, & Y \times \{ 0 \leq t \leq T \}, \\
		\bar{u} = f, \quad \bar{u}_t = 0, & Y \times \{ t = 0 \},
	\end{cases}
	\label{eq:nr_1d_hom_pde}
\end{equation}
where $\bar{A}$ is given by the \emph{harmonic average} of $A(x,y)$ over one $Y$-period,
\begin{equation}
	\bar{A}(x) = \int_{0}^{1} \frac{\D{y}}{A(x,y)},
	\label{eq:harmonicavg}
\end{equation}
and $x$ being held fixed. 

%%%%%%%%%%%%%%%%%%%%%%%%%%%%%%%%%%%%%%%%%%%%%%%%%%%%%%%%%%%%%%%%%%%%%%%%%%%%%%%%%%%%%%%%%%%%%%%%%%%%%%%%%%%%%%%%%%%%%%%%%%%%%%%%%%%%
\subsubsection{Example one} \label{section:example-one}
%%%%%%%%%%%%%%%%%%%%%%%%%%%%%%%%%%%%%%%%%%%%%%%%%%%%%%%%%%%%%%%%%%%%%%%%%%%%%%%%%%%%%%%%%%%%%%%%%%%%%%%%%%%%%%%%%%%%%%%%%%%%%%%%%%%%

The first wave propagation problem we choose $A^{\varepsilon}$ and $f$ as
\begin{equation}
	\begin{cases}
		A^{\varepsilon}(x) = A(x/\varepsilon), \qquad A(y) = 1.1 + \sin 2 \pi y, \\
		f(x) = \exp(-(x-x_0)^2/\sigma^2), \qquad x_0 = 0.5, \quad \sigma = 0.1.
	\end{cases}
	\label{eq:nr_1d_a_func}
\end{equation}
We can compute $\bar{A}$ from \eqref{eq:harmonicavg} with techniques from complex analysis
\begin{equation}
	\bar{A} = \sqrt{\frac{21}{100}} = 0.458257569495584\ldots
\end{equation}
We will solve \eqref{eq:nr_1d_hom_pde} with a fully resolved discretization or direct numerical simulation (DNS), discretized
homogenized solution (HOM) and our HMM method (HMM).  We have used $\varepsilon = 0.01$, $\eta = 10 \varepsilon$.  In
Figure~\ref{fig:1d_superimposed} we show a snapshot of the solutions these methods after time $T=1$.  We use a kernel (same in both
time and space) $K \in \mathbb{K}^{5,6}$, that is $K$ has 5 zero moments and is 6 times continuously differentiable. 
\begin{figure}[tbp]
	\centering
	\includegraphics{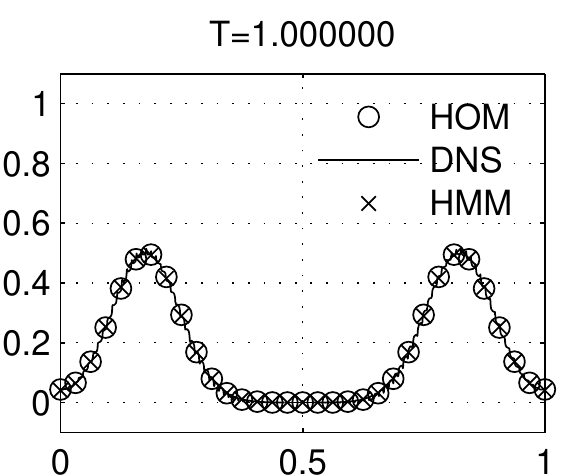}
	\includegraphics{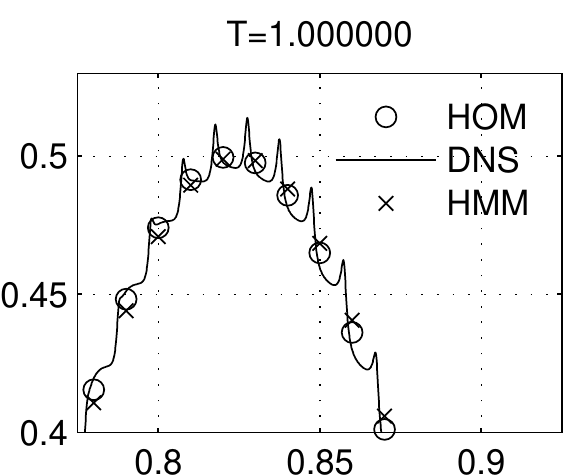}
	\caption{A snapshot of two super imposed solutions to \eqref{eq:nr_1d_pde} together with a zoomed section.}
	\label{fig:1d_superimposed}
\end{figure}

%%%%%%%%%%%%%%%%%%%%%%%%%%%%%%%%%%%%%%%%%%%%%%%%%%%%%%%%%%%%%%%%%%%%%%%%%%%%%%%%%%%%%%%%%%%%%%%%%%%%%%%%%%%%%%%%%%%%%%%%%%%%%%%%%%%%
\subsubsection{Example two} \label{section:example-two}
%%%%%%%%%%%%%%%%%%%%%%%%%%%%%%%%%%%%%%%%%%%%%%%%%%%%%%%%%%%%%%%%%%%%%%%%%%%%%%%%%%%%%%%%%%%%%%%%%%%%%%%%%%%%%%%%%%%%%%%%%%%%%%%%%%%%

We now consider a variation of \eqref{eq:nr_1d_pde} where $A^{\varepsilon}$ is defined as
\begin{equation}
	A^{\varepsilon}(x) = A(x,x/\varepsilon), \qquad A(x,y) = 1.1 + \frac{1}{2} \left( \cos 2 \pi x + \sin 2 \pi y \right).
	\label{eq:nr_1d_multiscale_pde}
\end{equation}
The homogenized operator $\bar{A}$ will not be constant but a function with explicit $x$ dependence.
We compute analytically $\bar{A}(x)$ to be
\begin{equation}
	\bar{A}(x) = \sqrt{\alpha(x)^2 - \beta^2} \qquad \alpha(x) = 1.1 + \frac{1}{2} \cos 2 \pi x, \quad \beta = \frac{1}{2}.
\end{equation}
For this experiment we use $\varepsilon = 0.01$, $K = 2 H$, $H = 3.33 \cdot 10^{-3}$.
For the micro problem we use $k/h = 0.5$ and $h = \varepsilon/64$. The kernel from $\mathbb{K}^{9,9}$.
The small $H$ is to lessen the effect of the numerical dispersion.
We show results from $T=1$ in Figure~\ref{fig:1d_superimposed_varying}. 

\begin{figure}[tbp]
	\centering
	\includegraphics{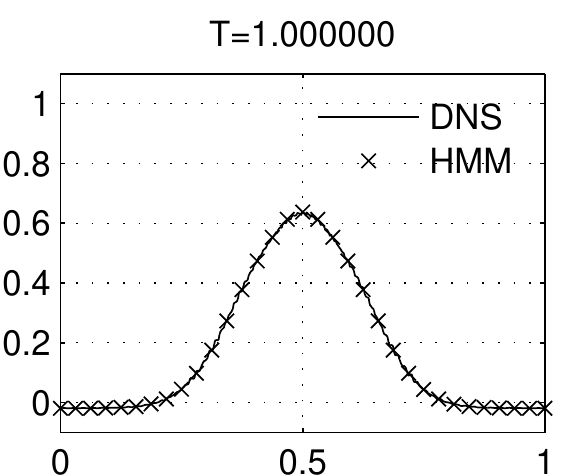}
	\includegraphics{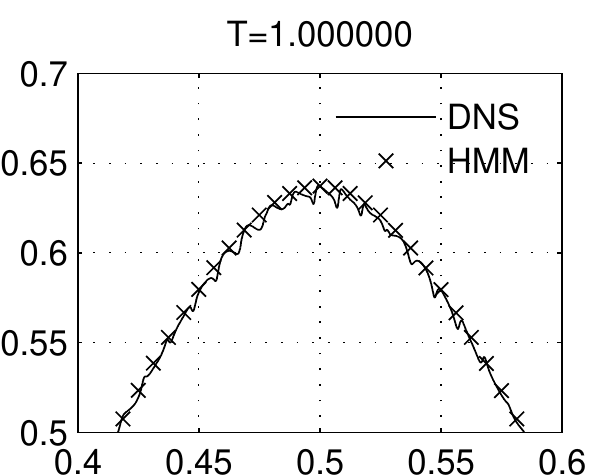}
	\caption{A snapshot of two super imposed solutions to \eqref{eq:nr_1d_multiscale_pde} together with a zoomed section.}
	\label{fig:1d_superimposed_varying}
\end{figure}

%%%%%%%%%%%%%%%%%%%%%%%%%%%%%%%%%%%%%%%%%%%%%%%%%%%%%%%%%%%%%%%%%%%%%%%%%%%%%%%%%%%%%%%%%%%%%%%%%%%%%%%%%%%%%%%%%%%%%%%%%%%%%%%%%%%%
\subsubsection{Example three} \label{section:example-three}
%%%%%%%%%%%%%%%%%%%%%%%%%%%%%%%%%%%%%%%%%%%%%%%%%%%%%%%%%%%%%%%%%%%%%%%%%%%%%%%%%%%%%%%%%%%%%%%%%%%%%%%%%%%%%%%%%%%%%%%%%%%%%%%%%%%%

In the last one-dimensional example the macro equation is unknown, i.e. homogenization does not provide $\bar{A}$. We
define $A^{\varepsilon}$ as a sum of many micro-scale oscillations
\begin{equation}
	\left\{
	\begin{aligned}
		& A^{\varepsilon}(x) = 1.1 + \frac{1}{5} \sum_{i=1}^{5} \sin 2 \pi \frac{x}{\varepsilon_i}, & \varepsilon_i = \frac{1}{90 + 5 (i-1)}, \\
		& f(x) = \exp(-(x-x_0)^2/\sigma^2), & x_0 = 0.5, \quad \sigma = 0.1.
	\end{aligned}
	\right.
\end{equation}
%A plot of $A^{\varepsilon}$ is shown in Figure \ref{fig:1d_material_varying2}.
A plot of $A^{\varepsilon}$ is shown in Figure \ref{fig:1d_superimposed_varying2}.
The numerical parameters for the macro-solver (HMM and homogenized) are $H = 3.33 \cdot 10^{-3}$, $K = 0.5 H$.
The micro solver uses $\tau = 10 \varepsilon_3$, $\eta = \varepsilon_3$, $h = \varepsilon_3/64$ and $k = 0.5 h$.
%\begin{center}
%	$\begin{array}{|cccccc|} \hline
%		H     & K    & \eta        & \tau           & h                & k    \\
%		1/300 & 0.5H & \varepsilon & 10 \varepsilon & \varepsilon_3/64 & 0.5 h  \\ \hline
%	\end{array}$
%\end{center}
The kernel $K$ used, for both time and space, is $K \in \mathbb{K}^{5,6}$. 
The results are shown in Figure \ref{fig:1d_superimposed_varying2}.

\begin{figure}[tbp]
	\centering
	\includegraphics{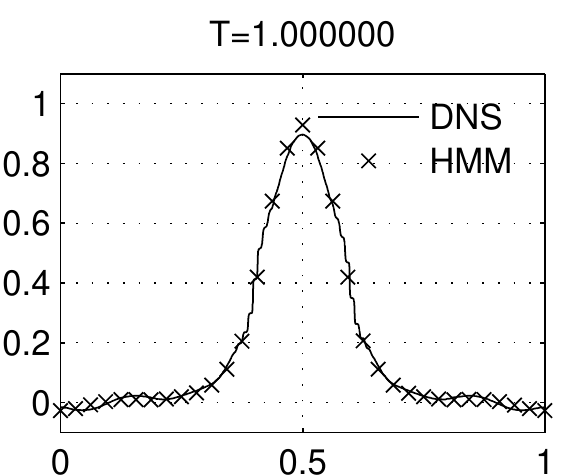}
	\includegraphics{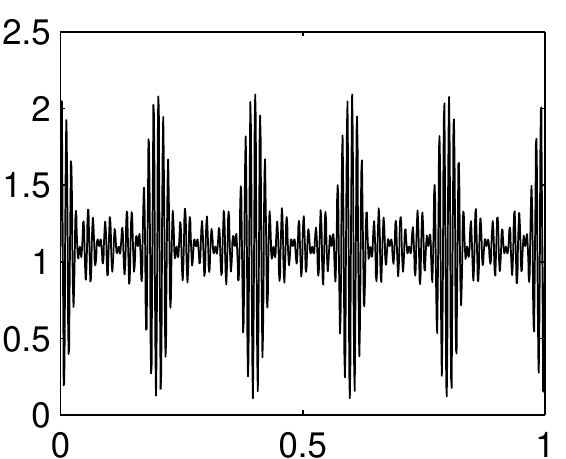}
	\caption{A snapshot of a direction solution to \eqref{eq:nr_1d_multiscale_pde} and the HMM solution (left)
	together with the material coefficient $A^{\varepsilon}$ from example \ref{section:example-three} (right).}
	\label{fig:1d_superimposed_varying2}
\end{figure}

%\begin{figure}[tbp]
%	\centering
%	\includegraphics[width=0.49\linewidth]{1d_material_varying2}
%	\caption{The material coefficient $A^{\varepsilon}$ in example \ref{section:example-three}.}
%	\label{fig:1d_material_varying2}
%\end{figure}

%%%%%%%%%%%%%%%%%%%%%%%%%%%%%%%%%%%%%%%%%%%%%%%%%%%%%%%%%%%%%%%%%%%%%%%%%%%%%%%%%%%%%%%%%%%%%%%%%%%%%%%%%%%%%%%%%%%%%%%%%%%%%%%%%%%%
\subsection{2D results}
%%%%%%%%%%%%%%%%%%%%%%%%%%%%%%%%%%%%%%%%%%%%%%%%%%%%%%%%%%%%%%%%%%%%%%%%%%%%%%%%%%%%%%%%%%%%%%%%%%%%%%%%%%%%%%%%%%%%%%%%%%%%%%%%%%%%

In this section we present the numerical results for a two dimensional wave 
propagation problem over the unit square $Y = [0,1] \times [0,1]$.

%%%%%%%%%%%%%%%%%%%%%%%%%%%%%%%%%%%%%%%%%%%%%%%%%%%%%%%%%%%%%%%%%%%%%%%%%%%%%%%%%%%%%%%%%%%%%%%%%%%%%%%%%%%%%%%%%%%%%%%%%%%%%%%%%%%%
\subsubsection{Example four} \label{section:example-four}
%%%%%%%%%%%%%%%%%%%%%%%%%%%%%%%%%%%%%%%%%%%%%%%%%%%%%%%%%%%%%%%%%%%%%%%%%%%%%%%%%%%%%%%%%%%%%%%%%%%%%%%%%%%%%%%%%%%%%%%%%%%%%%%%%%%%

We define $A^{\varepsilon}(x)$ by the diagonal matrix,
\begin{equation}
	\begin{cases}
		A^{\varepsilon}(x) = \diag(a^{\varepsilon}(x), a^{\varepsilon}(x)) \\
		a^{\varepsilon}(x) = a(x/\varepsilon), \quad a(y) = 1.1 + \sin 2 \pi y_1.
	\end{cases}
\end{equation}
The corresponding homogenized matrix $\bar{A}$ in \eqref{eq:introduction:wavebar},
\begin{equation}
	\bar{A} = \diag(\sqrt{0.21}, 1.1),
\end{equation}
and as in 1D the initial data $f$ is defined as a Gaussian,
\begin{equation}
	\begin{cases}
		f(x) = \exp(-\|x-x_0\|^2_2/\sigma^2), \\
		x_0 =  [0.5 \quad 0.5], \quad \sigma = 0.1. \\
	\end{cases}
\end{equation}
We use the exponential kernel $K_{\text{exp}} \in \mathbb{K}^{1,\infty}$.
We let $T=1$ and the scale parameter $\varepsilon$ is set to $0.01$.
The macro scheme uses $H = 3.33 \cdot 10^{-3}$ and $K = 0.5 H$.
The micro scheme uses $h = \varepsilon/64$ and $k = 0.5 h$.
We show the numerical results in Figure \ref{eq:2d_dns}, \ref{eq:2d_hom} and \ref{eq:2d_hmm}.

\begin{figure}[tbp]
	\centering
	\includegraphics{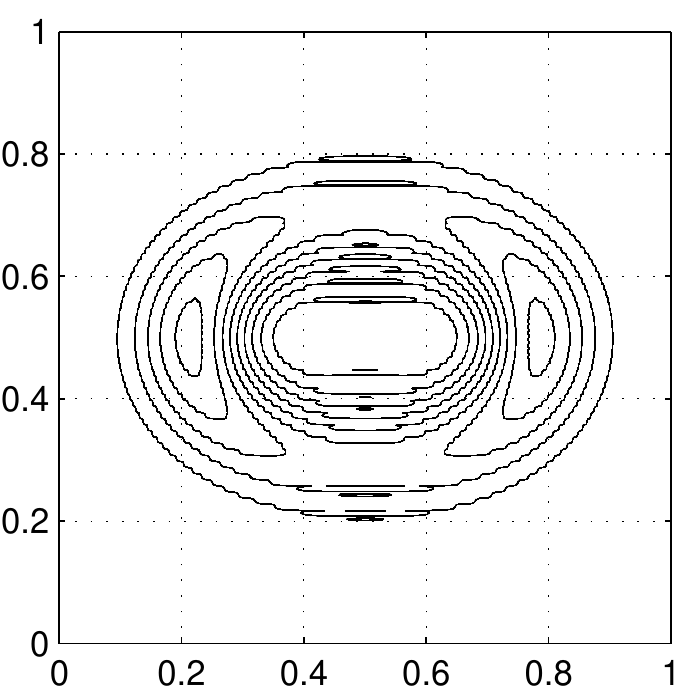}
	\caption{Full numerical simulation when $A^{\varepsilon}$ has only a fast scale.}
	\label{eq:2d_dns}
\end{figure}

\begin{figure}[tbp]
	\centering
	\includegraphics{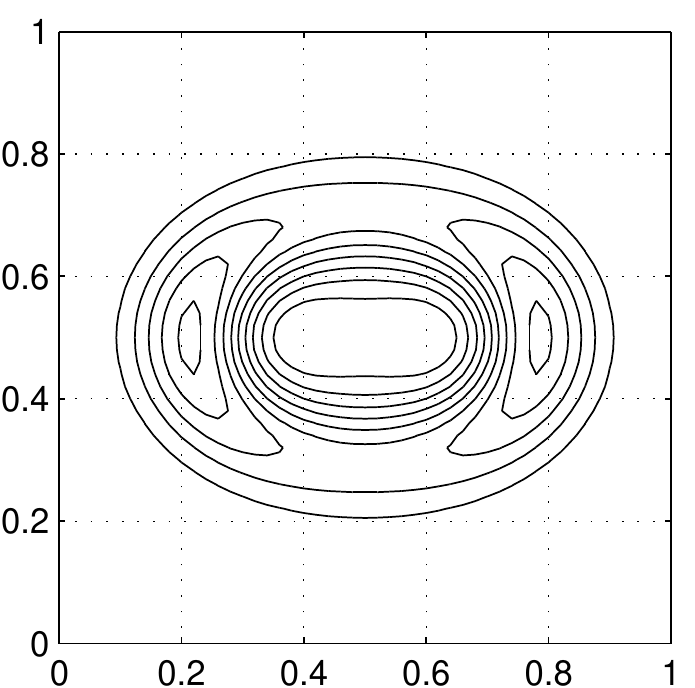}
	\caption{Direct solution of the homogenized equation when $\bar{A}$ is constant.}
	\label{eq:2d_hom}
\end{figure}

\begin{figure}[tbp]
	\centering
	\includegraphics{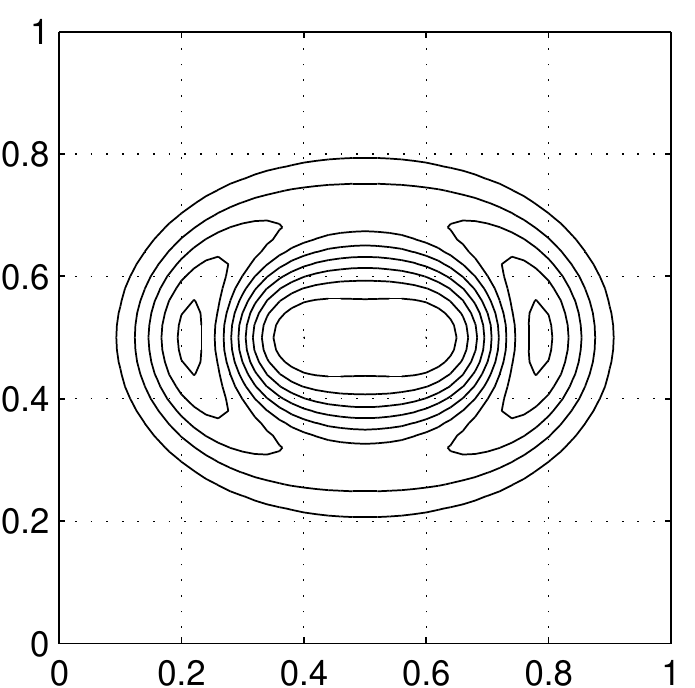}
	\caption{HMM approach, when $A^{\varepsilon}$ has only fast scales.}
	\label{eq:2d_hmm}
\end{figure}

%%%%%%%%%%%%%%%%%%%%%%%%%%%%%%%%%%%%%%%%%%%%%%%%%%%%%%%%%%%%%%%%%%%%%%%%%%%%%%%%%%%%%%%%%%%%%%%%%%%%%%%%%%%%%%%%%%%%%%%%%%%%%%%%%%%%
\subsubsection{Example five} \label{section:example-five}
%%%%%%%%%%%%%%%%%%%%%%%%%%%%%%%%%%%%%%%%%%%%%%%%%%%%%%%%%%%%%%%%%%%%%%%%%%%%%%%%%%%%%%%%%%%%%%%%%%%%%%%%%%%%%%%%%%%%%%%%%%%%%%%%%%%%

We let $A^{\varepsilon}(x)$ be defined by the diagonal matrix,
\begin{equation}	
	\begin{cases}
		A^{\varepsilon}(x) = \diag\left( a^{\varepsilon}(x), a^{\varepsilon}(x) \right) \\
		a^{\varepsilon}(x) = a(x,x/\varepsilon), \quad a(x,y) = 1.1 + \frac{1}{2}( \sin 2 \pi x_1 + \sin 2 \pi y_1). 
	\end{cases}
	\label{eq:nr_2d_pde_coef_varying}
\end{equation}
and the corresponding homogenized matrix $\bar{A}$ in \eqref{eq:introduction:wavebar},
\begin{equation}
	\begin{cases}
		\bar{A}(x) = \diag\left( \bar{a}(x), 1.1 \right), \\
		\bar{a}(x) = \sqrt{\alpha(x)^2 - \beta^2}, \quad \alpha(x) = 1.1 + 0.5 \sin 2 \pi x_1, \quad \beta = 0.5. \\
	\end{cases}
	\label{eq:nr_2d_hom_coef_varying}
\end{equation}
The numerical parameters are chosen the same as in example \ref{section:example-four}.
%We show the numerical results in Figures \ref{fig:2d_dns_varying}, \ref{fig:2d_hom_varying} and \ref{fig:2d_hmm_varying}.
We show the numerical results in Figures \ref{fig:2d_dns_varying} and \ref{fig:2d_hmm_varying}.

\begin{figure}[tbp]
	\centering
	\includegraphics{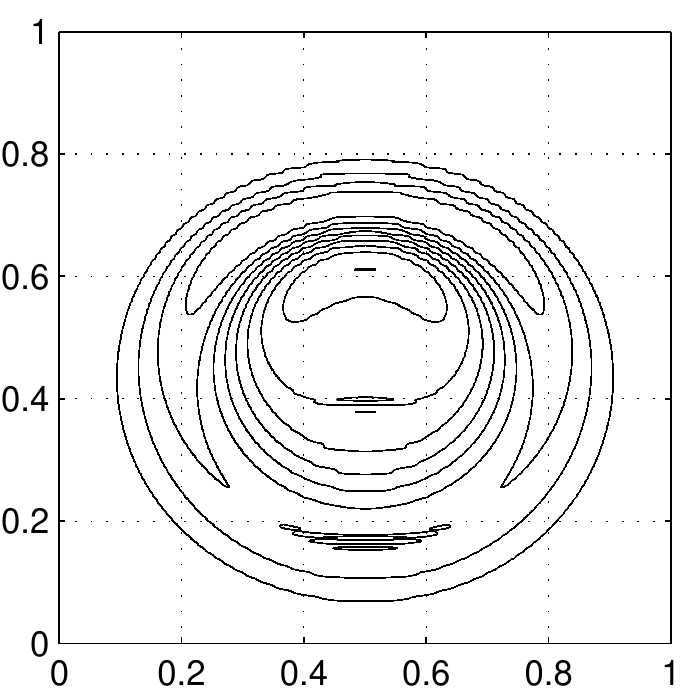}
	\caption{Full numerical simulation and $A^{\varepsilon}$ is defined by \eqref{eq:nr_2d_pde_coef_varying}.}
	\label{fig:2d_dns_varying}
\end{figure}

%\begin{figure}[tbp]
%	\centering
%	\includegraphics{2d_hom_varying}
%	\caption{Full numerical simulation of the homogenized Equation \eqref{eq:nr_2d_hom_coef_varying} with $\bar{A}$ defined by \eqref{eq:nr_2d_pde_coef_varying}.}
%	\label{fig:2d_hom_varying}
%\end{figure}

\begin{figure}[tbp]
	\centering
	\includegraphics{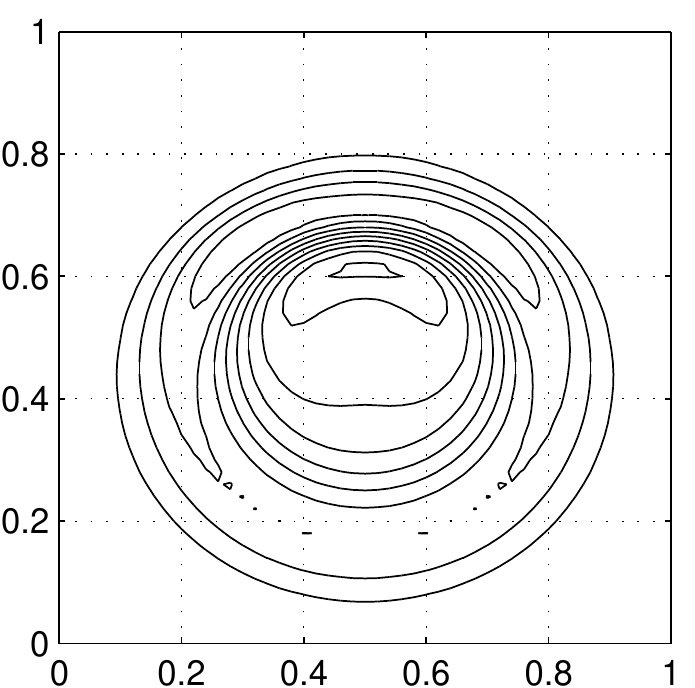}
	\caption{HMM approach and $A^{\varepsilon}$ is defined by \eqref{eq:nr_2d_pde_coef_varying}.}
	\label{fig:2d_hmm_varying}
\end{figure}

%%%%%%%%%%%%%%%%%%%%%%%%%%%%%%%%%%%%%%%%%%%%%%%%%%%%%%%%%%%%%%%%%%%%%%%%%%%%%%%%%%%%%%%%%%%%%%%%%%%%%%%%%%%%%%%%%%%%%%%%%%%%%%%%%%%%
\subsection{3D results}
%%%%%%%%%%%%%%%%%%%%%%%%%%%%%%%%%%%%%%%%%%%%%%%%%%%%%%%%%%%%%%%%%%%%%%%%%%%%%%%%%%%%%%%%%%%%%%%%%%%%%%%%%%%%%%%%%%%%%%%%%%%%%%%%%%%%

Here we present numerical results for a wave propagation problem in three dimensions in a locally periodic media
over the box $Y=[0,1]^3$.

%%%%%%%%%%%%%%%%%%%%%%%%%%%%%%%%%%%%%%%%%%%%%%%%%%%%%%%%%%%%%%%%%%%%%%%%%%%%%%%%%%%%%%%%%%%%%%%%%%%%%%%%%%%%%%%%%%%%%%%%%%%%%%%%%%%%
\subsubsection{Example six} \label{section:example-six}
%%%%%%%%%%%%%%%%%%%%%%%%%%%%%%%%%%%%%%%%%%%%%%%%%%%%%%%%%%%%%%%%%%%%%%%%%%%%%%%%%%%%%%%%%%%%%%%%%%%%%%%%%%%%%%%%%%%%%%%%%%%%%%%%%%%%

In this three dimensional problem $A^{\varepsilon}(x)$ is a diagonal matrix
\begin{equation}	
	\begin{cases}
		A^{\varepsilon}(x) = \diag\left(a^{\varepsilon}(x), a^{\varepsilon}(x), a^{\varepsilon}(x)\right), \\
		a^{\varepsilon}(x) = a(x/\varepsilon), \quad a(y) = 1.1 + \sin 2 \pi y_1, \\
	\end{cases}
	\label{eq:nr_longtime_pde}
\end{equation}
and the corresponding homogenized matrix $\bar{A}$ in \eqref{eq:introduction:wavebar} is
\begin{equation}
	\bar{A}(x) = \diag\left( \sqrt{0.21}, 1.1, 1.1 \right), \\
\end{equation}
and the initial data $f$ is defined as a Gaussian,
\begin{equation}
	\begin{cases}
		f(x) = \exp(-\|x-x_0\|^2_2/\sigma^2), \\
		x_0 = [ 0.5 \quad 0.5 \quad 0.5 ], \quad \sigma = 0.1. \\
	\end{cases} 
\end{equation}
In this experiment we have used $\varepsilon = 0.01$. 
The homogenized simulation uses 
$T = 0.25$, $H = 0.05$, $K = 0.25 H$.
The HMM solver uses 
$T = 0.25$, $H = 0.05$, $K = 0.25 H$ on the macro solver.
%% Micro problem
The micro solver uses 
$\eta = \varepsilon$, 
$\tau = 5 \varepsilon$, 
$h = \varepsilon/64$, 
$k = 0.3 h$ and a polynomial kernel $K \in \mathbb{K}^{9,9}$.
The results are presented in Figure \ref{fig:3d}.

\begin{remark}
	Due to the vast computational expense to use DNS we are unable to show DNS results.
\end{remark}

\begin{figure}[tbp]
	\centering
	\includegraphics{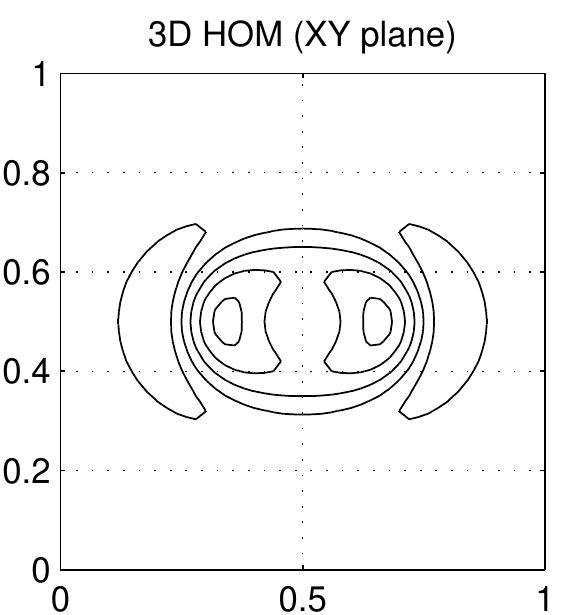} \includegraphics{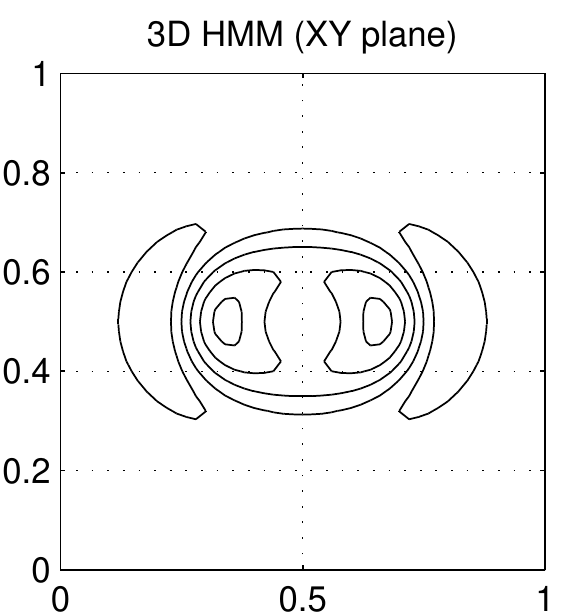}
	\includegraphics{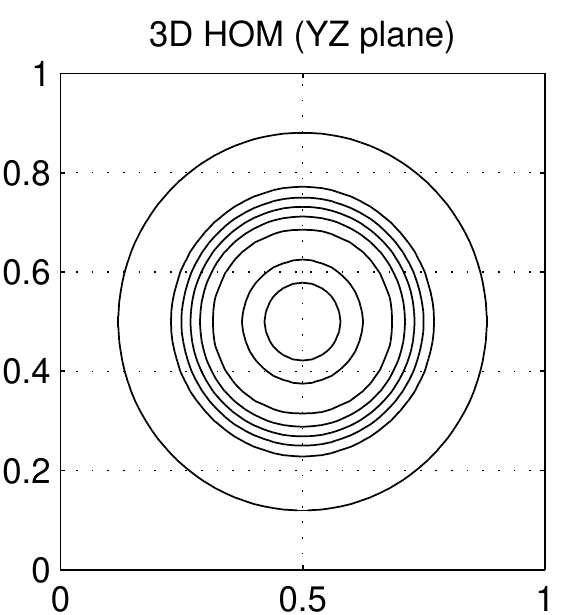} \includegraphics{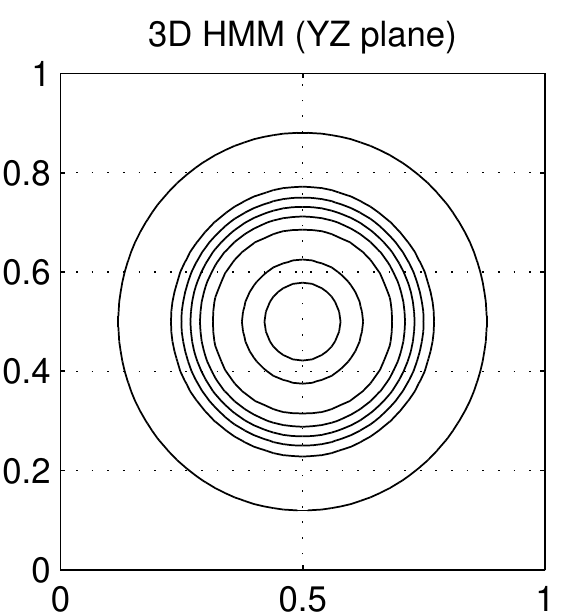}
	\caption{Three dimensional solutions of the homogenized equation and by using the HMM technique.}
	\label{fig:3d}
\end{figure}

%%%%%%%%%%%%%%%%%%%%%%%%%%%%%%%%%%%%%%%%%%%%%%%%%%%%%%%%%%%%%%%%%%%%%%%%%%%%%%%%%%%%%%%%%%%%%%%%%%%%%%%%%%%%%%%%%%%%%%%%%%%%%%%%%%%%
\subsection{Long time example} \label{section:example-longtime}
%%%%%%%%%%%%%%%%%%%%%%%%%%%%%%%%%%%%%%%%%%%%%%%%%%%%%%%%%%%%%%%%%%%%%%%%%%%%%%%%%%%%%%%%%%%%%%%%%%%%%%%%%%%%%%%%%%%%%%%%%%%%%%%%%%%%

We finally show a problem of the same form as example \ref{section:example-one}, but we will solve it for $T = \mathcal{O}(\varepsilon^{-2})$. 
In \cite{santosa1991} it was shown that the effective equation in this long time regime is of the form,
\begin{equation}
	\begin{cases}
		u_{tt} - \bar{A} u_{xx} - \beta \varepsilon^2 u_{xxxx} = 0, & Y \times \{ 0 \leq t \leq T \}, \\
		u = f, \quad u_t = 0, & Y \times \{ t = 0 \}.
	\end{cases}
	\label{eq:longtime-effective}
\end{equation}
This is still on the same flux form as assumed in \eqref{eq:hmm:wave} with $F = \bar{A} u_x + \beta \varepsilon^2 u_{xxx}$. 
Therefore, it turns out that we only need to make the HMM process a little bit more accurate for long time computations. 
The modifications needed are:
\begin{itemize}
	\item Initial data in micro solver needs to be of higher order. We use a third order polynomial to approximate the higher
		macro derivatives.
	\item The integration kernel needs to be smoother to give more accurate $F$ (error less than $\mathcal{O}(\varepsilon^2)$)
		in order to capture the correct dispersion relationship, i.e. $(\varepsilon/\eta)^q < \varepsilon^2$. This implies
		also that:
	\item The micro box needs to be a little bigger, $\tau, \eta \sim \varepsilon^{1-2/q}$, where $q$ 
		is defined in \eqref{eq:hmm:kernel}.
\end{itemize}
We present the numerical computations in Figure \ref{fig:1d_longtime_dns_hom}.

\begin{figure}[tbp]
	\begin{center}
		\includegraphics{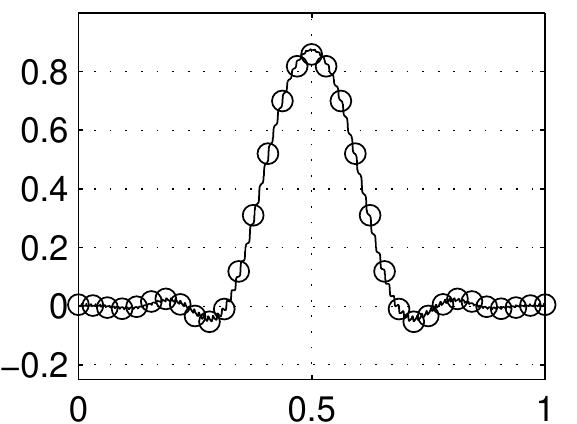} 
		\includegraphics{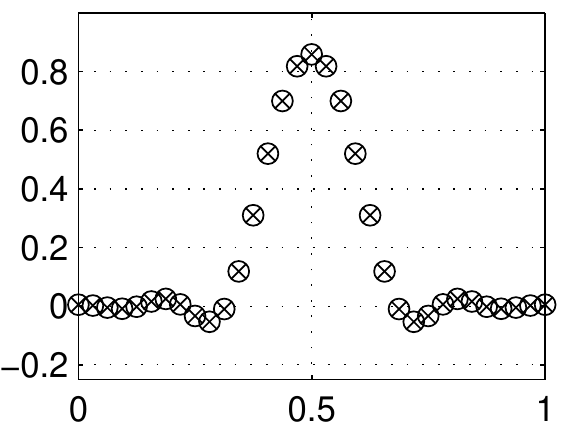}
	\end{center}
	\caption{1D longtime DNS simulation (thin line) compared to a finite difference solution of the effective equation
	\eqref{eq:longtime-effective} (circles) and a HMM solution (crosses).}
	\label{fig:1d_longtime_dns_hom}
\end{figure}

%%%%%%%%%%%%%%%%%%%%%%%%%%%%%%%%%%%%%%%%%%%%%%%%%%%%%%%%%%%%%%%%%%%%%%%%%%%%%%%%%%%%%%%%%%%%%%%%%%%%%%%%%%%%%%%%%%%%%%%%%%%%%%%%%%%%
\section{Conclusions} \label{section:conclusions}
%%%%%%%%%%%%%%%%%%%%%%%%%%%%%%%%%%%%%%%%%%%%%%%%%%%%%%%%%%%%%%%%%%%%%%%%%%%%%%%%%%%%%%%%%%%%%%%%%%%%%%%%%%%%%%%%%%%%%%%%%%%%%%%%%%%%

We have developed and analyzed numerical methods for multi-scale wave equations with oscillatory coefficients. The methods are based
on the framework of the heterogeneous multi-scale method (HMM) and have substantially lower computational complexity than standard
discretization algorithms. Convergence proofs for finite time approximation are presented in the case of periodic coefficients in
multiple dimensions. Numerical experiments in one, two and three spatial dimensions show the accuracy and efficiency of the new
techniques.  Finally we explored simulation over very long time intervals. The effective equation for very long time is different
from the finite time homogenized equation. Dispersive effects enter, and the effective equation must be modified~\cite{santosa1991}.
It is interesting to note that our HMM approach with just minor modifications accurately captures these dispersive phenomena.

%%%%%%%%%%%%%%%%%%%%%%%%%%%%%%%%%%%%%%%%%%%%%%%%%%%%%%%%%%%%%%%%%%%%%%%%%%%%%%%%%%%%%%%%%%%%%%%%%%%%%%%%%%%%%%%%%%%%%%%%%%%%%%%%%%%%
\appendix
%%%%%%%%%%%%%%%%%%%%%%%%%%%%%%%%%%%%%%%%%%%%%%%%%%%%%%%%%%%%%%%%%%%%%%%%%%%%%%%%%%%%%%%%%%%%%%%%%%%%%%%%%%%%%%%%%%%%%%%%%%%%%%%%%%%%

%%%%%%%%%%%%%%%%%%%%%%%%%%%%%%%%%%%%%%%%%%%%%%%%%%%%%%%%%%%%%%%%%%%%%%%%%%%%%%%%%%%%%%%%%%%%%%%%%%%%%%%%%%%%%%%%%%%%%%%%%%%%%%%%%%%%
\section{Numerical schemes} \label{section:schemes}
%%%%%%%%%%%%%%%%%%%%%%%%%%%%%%%%%%%%%%%%%%%%%%%%%%%%%%%%%%%%%%%%%%%%%%%%%%%%%%%%%%%%%%%%%%%%%%%%%%%%%%%%%%%%%%%%%%%%%%%%%%%%%%%%%%%%

We present a detailed description of the numerical schemes used in the macro and micro solvers.  
The schemes are designed for one, two, three dimensions and can be generalized to higher dimensions.
All the schemes are second order accurate in both time and space.

%%%%%%%%%%%%%%%%%%%%%%%%%%%%%%%%%%%%%%%%%%%%%%%%%%%%%%%%%%%%%%%%%%%%%%%%%%%%%%%%%%%%%%%%%%%%%%%%%%%%%%%%%%%%%%%%%%%%%%%%%%%%%%%%%%%%
\subsection{1D equation}
%%%%%%%%%%%%%%%%%%%%%%%%%%%%%%%%%%%%%%%%%%%%%%%%%%%%%%%%%%%%%%%%%%%%%%%%%%%%%%%%%%%%%%%%%%%%%%%%%%%%%%%%%%%%%%%%%%%%%%%%%%%%%%%%%%%%

The finite difference scheme on the macro level has the form
\begin{equation}
	\left\{
	\begin{aligned}
		& U^{n+1}_m = 2 U^n_m - U^{n-1}_m + K^2 Y^n_m, \\
		& Y^n_m = \frac{1}{H} \left( F^n_{m+\frac{1}{2}} - F^n_{m-\frac{1}{2}} \right), \\
		& F^n_{m \pm 1/2} = F(x_{m \pm 1/2},P^n_{m \pm 1/2}), \\
	\end{aligned}
	\right.
	\label{eq:1d_macro_scheme}
\end{equation}
where $P^n_{m - 1/2} = \frac{1}{H} \left( U^n_m - U^n_{m-1} \right)$ and $P^n_{m + 1/2} = \frac{1}{H} \left( U^n_{m+1} - U^n_m
\right)$.  The micro level scheme defined analogously:
\begin{equation} 
	\left\{
	\begin{aligned}
		& u^{n+1}_m = 2 u^n_m - u^{n-1}_m + k^2 y^n_m, \\
		& y^n_m = \frac{1}{h} \left( f^n_{m+1/2} - f^n_{m-1/2} \right), \\
		& f^n_{m+1/2} =  a_{m+\frac{1}{2}} \frac{u^n_{m+1} - u^n_m}{h}, \\
		& f^n_{m-1/2} =  a_{m-\frac{1}{2}} \frac{u^n_m - u^n_{m-1}}{h}.
	\end{aligned}
	\right.
	\label{eq:1d_micro_scheme}
\end{equation}

%%%%%%%%%%%%%%%%%%%%%%%%%%%%%%%%%%%%%%%%%%%%%%%%%%%%%%%%%%%%%%%%%%%%%%%%%%%%%%%%%%%%%%%%%%%%%%%%%%%%%%%%%%%%%%%%%%%%%%%%%%%%%%%%%%%%
\subsection{2D equation}
%%%%%%%%%%%%%%%%%%%%%%%%%%%%%%%%%%%%%%%%%%%%%%%%%%%%%%%%%%%%%%%%%%%%%%%%%%%%%%%%%%%%%%%%%%%%%%%%%%%%%%%%%%%%%%%%%%%%%%%%%%%%%%%%%%%%

The two dimensional problem is discretized with a scheme with the following schemes:
The finite difference scheme on the macro level
\begin{equation} 
	\left\{
	\begin{aligned}
		& U^{n+1}_m = 2 U^n_m - U^{n-1}_m + K^2 Y^n_m, \\
		& Y^n_m = \frac{1}{H} \left( F^{(1)}_{m+\frac{1}{2} e_1} - F^{(1)}_{m-\frac{1}{2} e_1} \right) + \frac{1}{H} \left( F^{(2)}_{m+\frac{1}{2} e_2} - F^{(2)}_{m-\frac{1}{2} e_2} \right), \\
		& F^{(k)}_{m \pm \frac{1}{2} e_k} = F(x_{m \pm \frac{1}{2} e_k}, P^n_{m \pm \frac{1}{2} e_k}), \quad t=t_n, \\
	\end{aligned}
	\right.
	\label{eq:2d_macro_scheme}
\end{equation}
where $P^n_{m + \frac{1}{2} e_2}$ is given by (see Figure \ref{fig:htwodisc5})
\begin{equation}
	P^n_{m + \frac{1}{2} e_2} = 
	\begin{bmatrix} 
		\frac{1}{2H}\left( \frac{U_{m+e_1} + U_{m+e_1+e_2}}{2} - \frac{U_{m-e_1} + U_{m-e_1+e_2}}{2} \right) &
		\frac{1}{H}\left( U_{m+e_2} - U_{m} \right)
	\end{bmatrix}.
	\label{eq:P_in_2d}
\end{equation}
and the other $P^n_{m \pm \frac{1}{2} e_k}$ are components defined analogously.
The micro level scheme is formulated as
\begin{equation}
	\left\{
	\begin{aligned}
		& u^{n+1}_m = 2 u^n_m - u^{n-1}_m + k^2 y^n_m \\
		& y^n_m = \frac{1}{h} \Bigl( f^{(1)}_{m+\frac{1}{2}e_1} - f^{(1)}_{m-\frac{1}{2}e_1} \Bigl) + \frac{1}{h} \Bigl( f^{(2)}_{m+\frac{1}{2}e_2} - f^{(2)}_{m-\frac{1}{2}e_2} \Bigr) \\
		& f^{(1)}_{m+\frac{1}{2} e_1} = \frac{a^{(11)}_{m+\frac{1}{2}e_1}}{h} \Bigl( u^n_{m+e_1} - u^n_{m} \Bigr) + \frac{a^{(12)}_{m+\frac{1}{2}e_1}}{2 h} \Bigl( \frac{u^n_{m+e_2} + u^n_{m+e_1+e_2}}{2} - \frac{u^n_{m-e_2} + u^n_{m+e_1-e_2}}{2} \Bigr) \\
		& f^{(1)}_{m-\frac{1}{2} e_1} = \frac{a^{(11)}_{m-\frac{1}{2}e_1}}{h} \Bigl( u^n_{m} - u^n_{m-e_1} \Bigr) + \frac{a^{(12)}_{m-\frac{1}{2}e_1}}{2 h} \Bigl( \frac{u^n_{m+e_2} + u^n_{m-e_1+e_2}}{2} - \frac{u^n_{m-e_2} + u^n_{m-e_1-e_2}}{2} \Bigr) \\
		& f^{(2)}_{m+\frac{1}{2} e_2} = \frac{a^{(21)}_{m+\frac{1}{2} e_2}}{2 h} \Bigl( \frac{u^n_{m+e_1} + u^n_{m+e_1+e_2}}{2} - \frac{u^n_{m-e_1} + u^n_{m-e_1+e_2}}{2} \Bigr) + \frac{a^{(22)}_{m+\frac{1}{2}e_2}}{h} \Bigl( u^n_{m+e_2} - u^n_{m} \Bigr) \\
		& f^{(2)}_{m-\frac{1}{2} e_2} = \frac{a^{(21)}_{m-\frac{1}{2} e_2}}{2 h} \Bigl( \frac{u^n_{m+e_1} + u^n_{m+e_1-e_2}}{2} - \frac{u^n_{m-e_1} + u^n_{m-e_1-e_2}}{2} \Bigr) + \frac{a^{(22)}_{m-\frac{1}{2}e_2}}{h} \Bigl( u^n_{m} - u^n_{m-e_2} \Bigr) \\
	\end{aligned}
	\right.
	\label{eq:2d_micro_scheme}
\end{equation}
When approximating $f^{(2)}_{m-\frac{1}{2}e_1}$ we take the average of $u^n_{m \pm e_2}$ and $u^n_{m + e_1 \pm e_2}$ to approximate
$u(x_{m + \frac{1}{2}e_1 \pm e_2},t^n)$. Then we use those two averages to approximate the $y$ derivate of $u$ at
$u(x_{m-\frac{1}{2}e_1})$.  The scheme is second order in both space and time.

%%%%%%%%%%%%%%%%%%%%%%%%%%%%%%%%%%%%%%%%%%%%%%%%%%%%%%%%%%%%%%%%%%%%%%%%%%%%%%%%%%%%%%%%%%%%%%%%%%%%%%%%%%%%%%%%%%%%%%%%%%%%%%%%%%%%
\subsection{3D equation}
%%%%%%%%%%%%%%%%%%%%%%%%%%%%%%%%%%%%%%%%%%%%%%%%%%%%%%%%%%%%%%%%%%%%%%%%%%%%%%%%%%%%%%%%%%%%%%%%%%%%%%%%%%%%%%%%%%%%%%%%%%%%%%%%%%%%

The macro scheme for the three dimensional problem is of the form
\begin{equation}
	\left\{
	\begin{aligned}
		& U^n_m = 2 U^n_m - U^{n-1}_m + K^2 Y^n_m, \\
		& Y^n_m = \frac{1}{H} \left( F^{(1,n)}_{m+\frac{1}{2} e_1} - F^{(1,n)}_{m-\frac{1}{2} e_1} \right)
		    \!+\! \frac{1}{H} \left( F^{(2,n)}_{m+\frac{1}{2} e_2} - F^{(2,n)}_{m-\frac{1}{2} e_2} \right)
		    \!+\! \frac{1}{H} \left( F^{(3,n)}_{m+\frac{1}{2} e_3} - F^{(3,n)}_{m-\frac{1}{2} e_3} \right),  \\
		& F^n_{m \pm \frac{1}{2} e_k} = F(x_{m \pm \frac{1}{2} e_k}, P^n_{m \pm \frac{1}{2} e_k}),
	\end{aligned}
	\right.
\end{equation}
where $P^n_{m+\frac{1}{2} e_3}$ is defined as,
\begin{equation} 
	P^n_{m+\frac{1}{2} e_3} = 
	\begin{bmatrix}
		\frac{1}{2H}\left( \frac{U_{m+e_1} + U_{m+e_1+e_3}}{2} - \frac{U_{m-e_1} + U_{m-e_1+e_3}}{2} \right)	\\
		\frac{1}{2H}\left( \frac{U_{m+e_2} + U_{m+e_2+e_3}}{2} - \frac{U_{m-e_2} + U_{m-e_2+e_3}}{2} \right) \\
		\frac{1}{H} \left( U_{m+e_3} - U_{m} \right)
	\end{bmatrix},
	\label{eq:3d_macro_scheme}
\end{equation}
and the other $P^n_{m \pm \frac{1}{2} e_k}$ defined analogously.
The micro level scheme is a second order accurate scheme defined analogous with the 2D scheme \eqref{eq:2d_micro_scheme}
\begin{gather} 
	\left\{
	\begin{split}
		& u^{n+1}_m = 2 u^n_m - u^{n-1}_m + k^2 y^n_m \\
		& y^n_m = \frac{1}{h} \Bigl( f^{(1)}_{m+\frac{1}{2}e_1} - f^{(1)}_{m-\frac{1}{2}e_1} \Bigl) + \frac{1}{h} \Bigl( f^{(2)}_{m+\frac{1}{2}e_2} - f^{(2)}_{m-\frac{1}{2}e_2} \Bigr) + \frac{1}{h} \Bigl(  f^{(3)}_{m+\frac{1}{2}e_3} - f^{(3)}_{m-\frac{1}{2}e_3} \Bigr) \\
		& \begin{split}
			f^{(1)}_{m+\frac{1}{2} e_1} = 
			    \frac{a^{(11)}_{m+\frac{1}{2}e_1}}{ h } \Bigl( u^n_{m+e_1} - u^n_{m} \Bigr) 
			  + \frac{a^{(12)}_{m+\frac{1}{2}e_1}}{2 h} \Bigl( \frac{u^n_{m+e_1+e_2} + u^n_{m+e_2}}{2} - \frac{u^n_{m+e_1-e_2} + u^n_{m-e_2}}{2} \Bigr) \\
			  + \frac{a^{(13)}_{m+\frac{1}{2}e_1}}{2 h} \Bigl( \frac{u^n_{m+e_1+e_3} + u^n_{m+e_3}}{2} - \frac{u^n_{m+e_1-e_3} + u^n_{m+e_3}}{2} \Bigr)
			\end{split} \\
		& \begin{split}
			f^{(1)}_{m-\frac{1}{2} e_1} = 
			    \frac{a^{(11)}_{m-\frac{1}{2}e_1}}{ h } \Bigl( u^n_{m} - u^n_{m-e_1} \Bigr) 
			  + \frac{a^{(12)}_{m-\frac{1}{2}e_1}}{2 h} \Bigl( \frac{u^n_{m+e_2} + u^n_{m-e_1+e_2}}{2} - \frac{u^n_{m-e_2} + u^n_{m-e_1-e_2}}{2} \Bigr) \\
			  + \frac{a^{(13)}_{m-\frac{1}{2}e_1}}{2 h} \Bigl( \frac{u^n_{m+e_3} + u^n_{m-e_1+e_3}}{2} - \frac{u^n_{m-e_3} + u^n_{m-e_1-e_3}}{2} \Bigr)
			\end{split} \\
		& \begin{split}
			f^{(2)}_{m+\frac{1}{2} e_2} = 
			    \frac{a^{(21)}_{m+\frac{1}{2}e_2}}{2 h} \Bigl( \frac{u^n_{m+e_1+e_2} + u^n_{m+e_1}}{2} - \frac{u^n_{m-e_1+e_2} + u^n_{m-e_1}}{2} \Bigr) 
			  + \frac{a^{(22)}_{m+\frac{1}{2}e_2}}{ h } \Bigl( u^n_{m+e_2} -u^n_{m} \Bigr) \\
			  + \frac{a^{(23)}_{m+\frac{1}{2}e_2}}{2 h} \Bigl( \frac{u^n_{m+e_2+e_3} + u^n_{m+e_3}}{2} - \frac{u^n_{m+e_2-e_3} + u^n_{m-e_3}}{2} \Bigr)
			\end{split} \\		
		& \begin{split}
			f^{(2)}_{m-\frac{1}{2} e_2} = 
			    \frac{a^{(21)}_{m-\frac{1}{2}e_2}}{2 h} \Bigl( \frac{u^n_{m+e_1} + u^n_{m+e_1-e_2}}{2} - \frac{u^n_{m-e_1} + u^n_{m-e_1-e_2}}{2} \Bigr) 
			  + \frac{a^{(22)}_{m-\frac{1}{2}e_2}}{ h } \Bigl( u^n_{m} -u^n_{m-e_2} \Bigr) \\
			  + \frac{a^{(23)}_{m-\frac{1}{2}e_2}}{2 h} \Bigl( \frac{u^n_{m+e_3} + u^n_{m-e_2+e_3}}{2} - \frac{u^n_{m-e_3} + u^n_{m-e_2-e_3}}{2} \Bigr)
			\end{split} \\
		& \begin{split}
			f^{(3)}_{m+\frac{1}{2} e_3} = 
			    & \frac{a^{(31)}_{m+\frac{1}{2}e_3}}{2 h} \Bigl( \frac{u^n_{m+e_1+e_3} + u^n_{m+e_1}}{2} - \frac{u^n_{m-e_1+e_3} + u^n_{m-e_1}}{2} \Bigr) \\
			  + & \frac{a^{(32)}_{m+\frac{1}{2}e_3}}{2 h} \Bigl( \frac{u^n_{m+e_2+e_3} + u^n_{m+e_2}}{2} - \frac{u^n_{m-e_2+e_3} + u^n_{m-e_2}}{2} \Bigr) 
			  +   \frac{a^{(33)}_{m+\frac{1}{2}e_3}}{ h } \Bigl( u^n_{m+e_3} - u^n_{m} \Bigr)
			\end{split} \\
		& \begin{split}
			f^{(3)}_{m-\frac{1}{2} e_3} = 
			    & \frac{a^{(31)}_{m-\frac{1}{2}e_3}}{2 h} \Bigl( \frac{u^n_{m+e_1} + u^n_{m+e_1-e_3}}{2} - \frac{u^n_{m-e_1} + u^n_{m-e_1-e_3}}{2} \Bigr) \\
			  + & \frac{a^{(32)}_{m-\frac{1}{2}e_3}}{2 h} \Bigl( \frac{u^n_{m+e_2} + u^n_{m+e_2-e_3}}{2} - \frac{u^n_{m-e_2} + u^n_{m-e_2-e_3}}{2} \Bigr) 
			  +   \frac{a^{(33)}_{m-\frac{1}{2}e_3}}{ h } \Bigl( u^n_{m} - u^n_{m-e_3} \Bigr)
			\end{split} 
	\end{split}
	\right.
	\label{eq:3d_micro_scheme}
\end{gather}

%%%%%%%%%%%%%%%%%%%%%%%%%%%%%%%%%%%%%%%%%%%%%%%%%%%%%%%%%%%%%%%%%%%%%%%%%%%%%%%%%%%%%%%%%%%%%%%%%%%%%%%%%%%%%%%%%%%%%%%%%%%%%%%%%%%%
%% bibliography
%%%%%%%%%%%%%%%%%%%%%%%%%%%%%%%%%%%%%%%%%%%%%%%%%%%%%%%%%%%%%%%%%%%%%%%%%%%%%%%%%%%%%%%%%%%%%%%%%%%%%%%%%%%%%%%%%%%%%%%%%%%%%%%%%%%%

\bibliographystyle{plain}
\bibliography{holst}

\end{document}